\documentclass[english,11pt]{article}

\usepackage{amssymb,amsmath,amsfonts,amsthm,rotating,setspace,graphicx,float}
\usepackage{amssymb,amsmath,amsthm,rotating,setspace}
\usepackage{enumerate}

\usepackage{latexsym}

\usepackage{epsfig}

\usepackage{color}

\newtheorem{theorem}{Theorem}[section]
\newtheorem{lemma}[theorem]{Lemma}

\newtheorem{cor}[theorem]{Corollary}

\theoremstyle{definition}
\newtheorem{de}[theorem]{Definition}

\theoremstyle{definition}
\newtheorem{remark}[theorem]{Remark}

\theoremstyle{proposition}

\numberwithin{equation}{section}

\def\eq#1{(\ref{#1})}

\def\N{\mathbb{N}}

\def\R{\mathbb{R}}

\begin{document}

\title{Classification of Radial Solutions for Semilinear Elliptic Systems with Nonlinear Gradient Terms}

\author{{\large Gurpreet Singh}\\School of Mathematical Sciences\\
 University College Dublin, Belfield\\
 Dublin 4, Ireland\\E-mail: {\tt gurpreet.singh@ucdconnect.ie}}

%\address{School of Mathematical Sciences, University College Dublin, Belfield, Dublin 4, Ireland}

%\email{gurpreet.singh@ucdconnect.ie}

%\subjclass[2010]{Primary 35B40; 35J75; Secondary 35J25; 35B51}

%\date{21 mai 2011}

%\keywords{Entire (large) solution; radial symmetry; asymptotic behavior, uniqueness}

\maketitle

\begin{abstract}
We are concerned with the classification of positive radial solutions for the system $\Delta u=v^p$, $\Delta v=f(|\nabla u|)$, where $p>0$ and $f\in C^1[0,\infty)$ is a nondecreasing function such that $f(t)>0$ for all $t>0$. We show that in the case where the system is posed in the whole space $\R^N$ such solutions exist if and only if  $\displaystyle\int_{1}^\infty \Big(\displaystyle \int_0^s F(t)dt  \Big)^{-p/(2p+1)} ds =\infty$. This is the counterpart of the Keller-Osserman condition for the case of single semilinear equation. Similar optimal conditions are derived in case where the system is posed in a ball of $\R^N$. If $f(t)=t^q$, $q>1$, using dynamical system techniques we are able to describe the behaviour of solutions at infinity (in case where the system is posed in the whole $\R^N$) or around the boundary (in case of a ball). 

\end{abstract}

%\tableofcontents\begin{proof}

\section{Introduction}

In 1950s Keller \cite{K1954} and Osserman \cite{O1957} obtained independently optimal 
conditions for the existence of a solution to the boundary blow-up problem
\begin{equation}\label{ko}
\left\{
\begin{aligned}
\Delta u&=f(u)&&\quad\mbox{ in }\Omega,\\
u&=\infty &&\quad\mbox{ on }\partial\Omega,
\end{aligned}
\right.
\end{equation}
where $\Omega\subset \R^N$ is a bounded domain and $f\in C^1[0,\infty)$ is a nonnegative increasing function. The condition on the boundary $\partial\Omega$ in \eq{ko} is understood as  $\lim_{x\to x_0} u(x)=\infty$ for all $x_0\in \partial\Omega$.
Keller and Osserman obtained that \eq{ko} has $C^2(\Omega)$ solutions if and only if
\begin{equation}\label{kocond}
\int_1^\infty\frac{ds}{\sqrt{F(s)}}<\infty\quad\mbox{ where }\; F(s)=\int_0^s f(t)dt.
\end{equation}
Interestingly, condition \eq{kocond} also appeared in other circumstances: it is related to the maximum principle for nonlinear elliptic inequalities. For instance, if $u\in C^2(\Omega)$ is nonnegative and satisfies $\Delta u\leq f(u)$ in $\Omega$, then, if $u$ vanishes at a point in $\Omega$, it must vanish everywhere in $\Omega$. We refer the reader to Vazquez \cite{V1984} and to Pucci, Serrin and Zou  \cite{PS1993,PS2004,PSZ1999} for various extensions of this result.

Problems related to boundary blow-up solutions have a long history and they can be traced back to at least a century ago when Bieberbach \cite{B1916} investigated such solutions for the equation $\Delta u=e^u$ in a planar domain. Since then, many techniques have been devised to deal with such solutions (see, e.g. \cite{GRbook2008,GRbook2012,Rbk} for an account on the progress on this topic). Boundary blow-up solutions for semilinear elliptic equations with nonlinear gradient terms have been only recently investigated (see for instance \cite{AGMQ2012,CPW2013,FQS2013,MMMR2011}).

In this paper we investigate a semilinear elliptic system featuring a mixture of power type nonlinearities and nonlinear gradient terms. More precisely, we shall be concerned with 
\begin{equation}\label{sys0}
\left\{
\begin{aligned}
\Delta u&=v^p&&\quad\mbox{ in }\Omega,\\
\Delta v&=f(|\nabla u|) &&\quad\mbox{ in }\Omega,
\end{aligned}
\right.
\end{equation}
where $\Omega\subset \R^N$ is either a ball centred at the origin or the whole space, $p>0$ is a real number and $f\in C^1[0,\infty)$ is a nondecreasing function such that $f(t)>0$ for all $t>0$. Our study will assume that $u$ and $v$ are positive radially symmetric solutions of \eq{sys0}. Note that we do not assume a priori any condition at the boundary for neither $u$ or $v$ but this will be needed in the course of our analysis as we shall be concerned with the classification of all solutions to \eq{sys0}.

If $\Omega$ is a ball, system \eq{sys0} was first considered by Diaz, Lazzo, and Schmidt in \cite{DLS2005}, in the case $p=1$ and $f(t)=t^2$. Such choice of exponent $p$ and function $f$ is related to the study of the dynamics of a viscous, heat-conducting fluid. The authors in \cite{DLS2005} obtained the existence of one positive solution and, in case of small dimensions, of one sign-changing solution that blows up at the boundary. Their study was further extended to time dependent systems in Diaz, Rakotoson, and Schmidt \cite{DRS2007, DRS2008}.

We shall first be concerned with the case where $\Omega$ is a ball. In such a situation we obtain that \eq{sys0} admits positive radially symmetric  solutions $(u,v)$ such that $u$ or $v$ (or both) blow up around $\partial \Omega$ if and only if
\begin{equation}\label{KOgradient}
\int_{1}^\infty\frac{ds}{\Big(\displaystyle \int_0^s F(t)dt  \Big)^{p/(2p+1)}} <\infty.
\end{equation}
This can be seen as the analogous condition to \eq{kocond} obtained by Keller \cite{HV1996} and Osserman \cite{MST1995} for \eq{ko}. We also provide a complete classification of radially symmetric solutions in such a case. Moreover, we shall obtain (see Theorem \ref{thm2} below) that the equation
$$
\Delta^2 u=f(|\nabla u|)\quad\mbox{ in }B_R
$$
has (not necessarily positive) radially symmetric solutions that blow up at the boundary $\partial B_R$ if and only if
$$
\int_1^\infty\frac{sds}{\displaystyle \Big(\int_0^{s}{F(t)}{dt}\Big)^{1/3}}=\infty
 \quad\mbox{and}\quad
\int_1^\infty\frac{ds}{\displaystyle \Big(\int_0^{s}{F(t)}{dt}\Big)^{1/3}}<\infty.
$$
If $f(t)=t^q$, $q\geq 1$, we are able to give the exact rate at which the components $u$ and $v$ blow up at the boundary. In such a setting we use dynamical systems tools for cooperative systems with negative divergence.

Further, condition \eq{KOgradient} appears again in the study of \eq{sys0} in the case $\Omega=\R^N$. Again when $f$ is a pure power type nonlinearity we shall be able to precisely describe the behaviour of solutions at infinity.

\section{Main results}

Let us first present the analysis of system \eq{sys0} in the case where $\Omega$ is a ball. Namely, we shall first investigate the system 
\begin{equation}\label{sys}
\left\{
\begin{aligned}
\Delta u&=v^p&&\quad\mbox{ in }B_R,\\
\Delta v&=f(|\nabla u|) &&\quad\mbox{ in }B_R,
\end{aligned}
\right.
\end{equation}
where $B_R\subset \R^N,$ $N\geq 2$ is the open ball of radius $R>0$ centred at the origin, $p>0$ and $f\in C^1[0,\infty)$ is a nondecreasing function such that $f(t)>0$ for all $t>0$. Let $F$ be the antiderivative of $f$ that vanishes at the origin (see \eq{kocond}).
Sometimes in this paper we shall complement the system \eq{sys} with one of the following boundary conditions:
\begin{itemize}
\item either
$
\mbox{ $u$ and $v$ are bounded in $B_R$;}
$
\item or
\begin{equation}\label{cond1}
u\mbox{ is bounded in $B_R$ and } \lim_{|x|\nearrow R}v(x)=\infty;
\end{equation}
\item or
\begin{equation}\label{cond2}
\lim_{|x|\nearrow R}u(x)=\lim_{|x|\nearrow R}v(x)=\infty.
\end{equation}
\end{itemize}

From the first equation of \eq{sys} it is easy to see that the situation  $\lim_{|x|\nearrow R}u(x)=\infty$ and $v$ is bounded in $B_R$ cannot occur.

\begin{theorem}\label{thm1}
We have:
\begin{enumerate}
\item[(i)] All positive radial solutions of \eq{sys} are bounded if and only if 
\begin{equation}\label{bounded}
\int_{1}^\infty\frac{ds}{\Big(\displaystyle \int_0^s F(t)dt  \Big)^{p/(2p+1)}} =\infty.
\end{equation}
\item[(ii)] There exists a positive radial solution $(u,v)$ of \eq{sys} satisfying \eq{cond1}  if and only if 
\begin{equation}\label{int1}
\int_{1}^\infty\frac{sds}{\Big(\displaystyle \int_0^s F(t)dt  \Big)^{p/(2p+1)}} <\infty.
\end{equation}
\item[(iii)] The exists a positive radial solution $(u,v)$ of \eq{sys} satisfying \eq{cond2} if and only if
\begin{equation}\label{int2}
\int_{1}^\infty\frac{ds}{\Big(\displaystyle \int_0^s F(t)dt  \Big)^{p/(2p+1)}} <\infty\quad\mbox{ and }
\int_{1}^\infty\frac{sds}{\Big(\displaystyle \int_0^s F(t)dt  \Big)^{p/(2p+1)}} =\infty.
\end{equation}
\end{enumerate}
\end{theorem}
By taking $f(t)=e^t$ and estimating the integrals in \eq{bounded}-\eq{int2} we find:
\begin{cor}\label{corblowup2}
Consider
\begin{equation}\label{eqexp}
\left\{
\begin{aligned}
\Delta u&=v^{p}&&\quad\mbox{ in }B_R,\\
\Delta v&= e^{|\nabla u|}&&\quad\mbox{ in }B_R.
\end{aligned}
\right.
\end{equation}
Then any solution of \eq{eqexp} is either bounded or satisfies \eq{cond1}.
\end{cor}
We now let $f(t)=t^q$, $q\geq 1$. From Theorem \ref{thm1} we obtain:
\begin{cor}\label{corblowup3}
Consider
\begin{equation}\label{eqtq}
\left\{
\begin{aligned}
\Delta u&=v^{p}&&\quad\mbox{ in }B_R,\\
\Delta v&=|\nabla u|^{q}&&\quad\mbox{ in }B_R,
\end{aligned}
\right.
\end{equation}
where $p>0$ and $q\geq 1$. Then we have:
\begin{enumerate}
\item [(i)] All positive radial solutions of \eq{eqtq} are bounded if and only if $$p\leq 1 \mbox{ and } 1\leq q\leq \frac{1}{p}.$$
\item [(ii)] There exists positive radial solutions of \eq{eqtq} satisfying \eq{cond1} if and only if $$q>2\Big(1+\frac{1}{p}\Big).$$
\item [(iii)] There exists positive radial solutions of \eq{eqtq} satisfying \eq{cond2} if and only if $$\frac{1}{p}< q\leq 2 \Big(1+\frac{1}{p}\Big).$$
\end{enumerate}
\end{cor}
The three regions $A$, $B$ and $C$ in the $pq$-plane that correspond to the cases (i), (ii) and (iii) in Corollary \ref{corblowup3} are depicted below.

\begin{figure}[!htb]
\centering
\includegraphics[scale=.45]{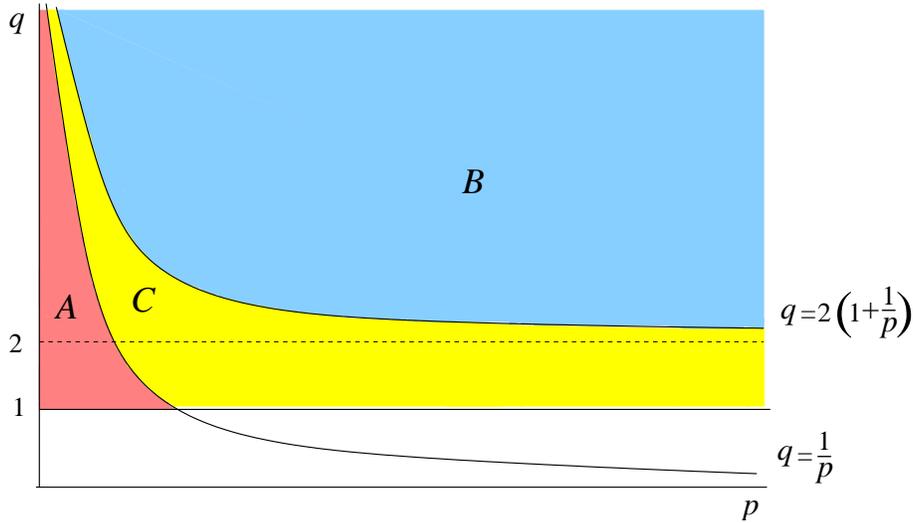}
\caption{The three regions described by Corollary \ref{corblowup3} }
%\label{fig:digraph}
\end{figure}

In the two pictures below we used MATLAB to plot the solution $(u,v)$ of system \eq{eqtq} for $q=3$ and $p=2$, $p=4$ and for various space dimensions $N$.

Our next result deals with the biharmonic problem that derives from \eq{sys} by taking $p=1$. In this case we are able to deduce optimal conditions for the existence of a boundary blow up solution. 

\begin{theorem}\label{thm2}
Let $R>0$. The problem 
\begin{equation}\label{sysb}
\left\{
\begin{aligned}
\Delta^2 u&=f(|\nabla u|)&&\quad\mbox{ in }B_R,\\
u&=\infty&&\quad\mbox{ on }\partial B_R,
\end{aligned}
\right.
\end{equation}
has radially symmetric solutions if and only if 
\begin{equation}\label{intc}
\int_1^\infty\frac{sds}{\displaystyle \Big(\int_0^{s}{F(t)}{dt}\Big)^{1/3}}=\infty
 \quad\mbox{and}\quad
\int_1^\infty\frac{ds}{\displaystyle \Big(\int_0^{s}{F(t)}{dt}\Big)^{1/3}}<\infty
\end{equation}
\end{theorem}

\begin{figure}[!htb]\label{fig2}
\centering
\includegraphics[height=7.7cm,width=17cm]{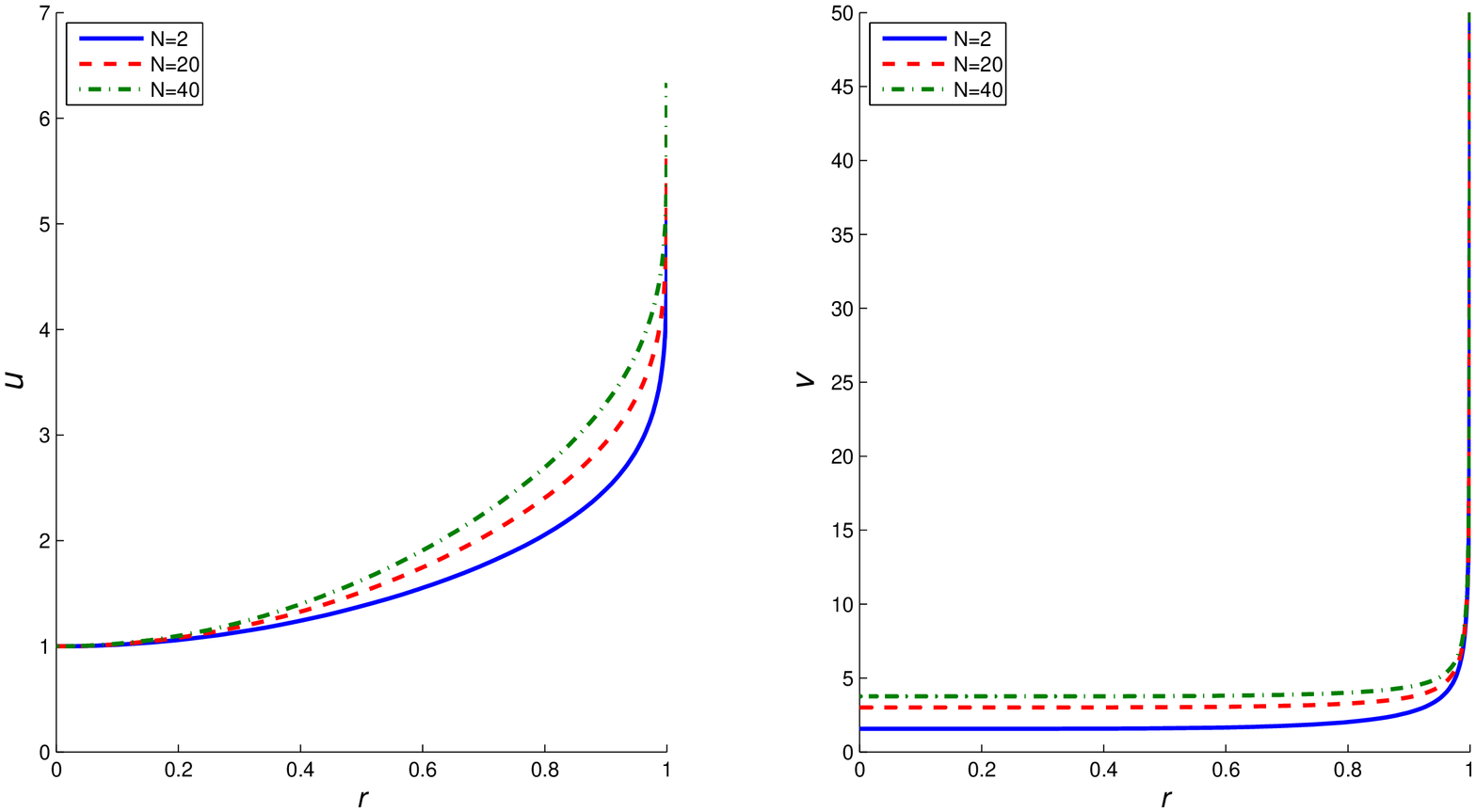}
\caption{The solution $u$ (left) and $v$ (right) for system \eq{eqtq} with boundary condition \eq{cond1} in the case $p=4$, $q=3$ and for various space dimensions $N=2$, $N=20$ and $N=40$. The graph of the $v$-component is restricted on the vertical axis to the interval $[0,50]$. }
\end{figure}
\begin{figure}[!htb]\label{fig3}
\centering
\includegraphics[height=7.7cm,width=17cm]{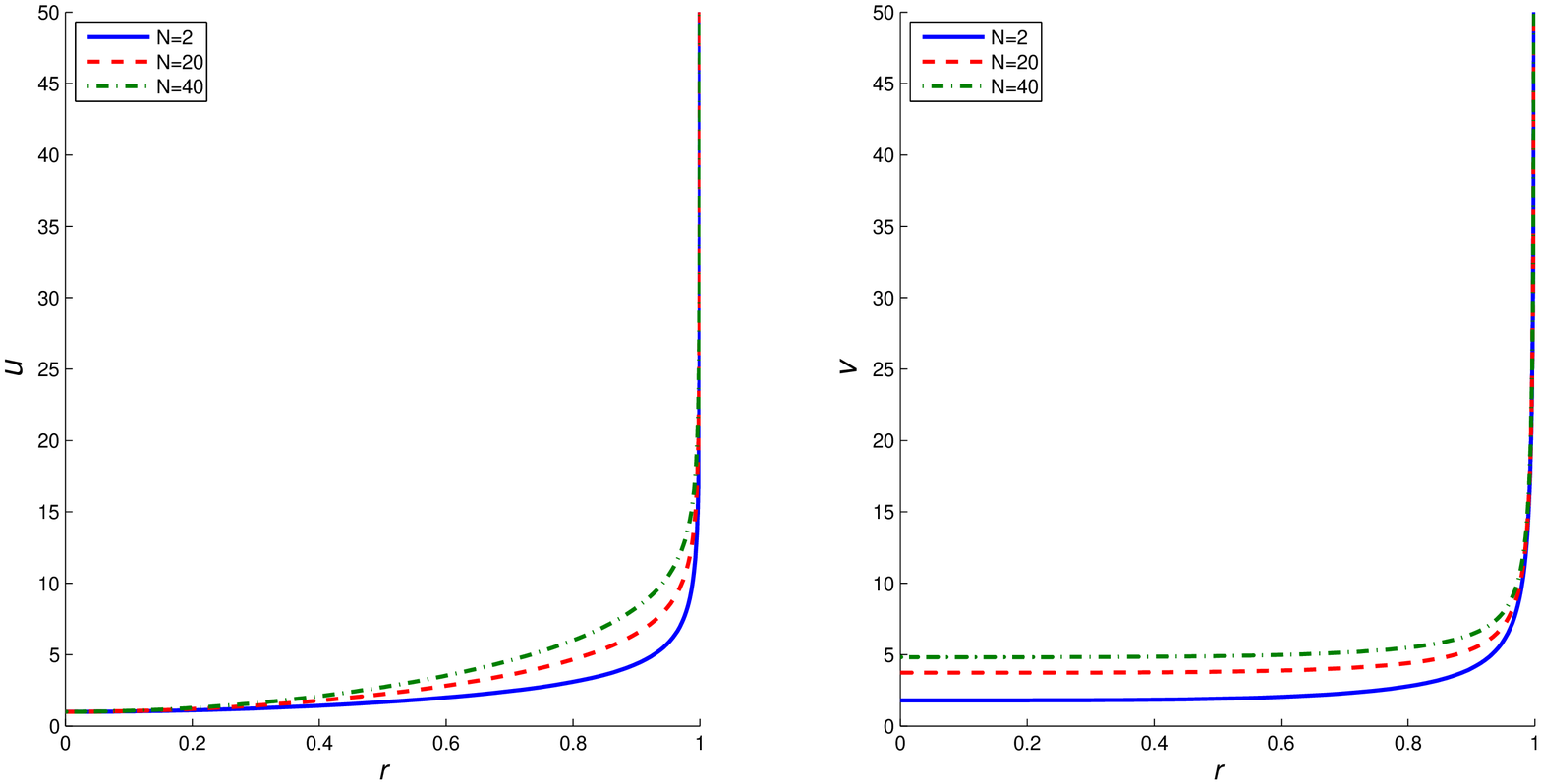}
\caption{The solution $u$ (left) and $v$ (right) for system \eq{eqtq} with boundary condition \eq{cond2} in the case $p=2$, $q=3$ and for various space dimensions $N=2$, $N=20$ and $N=40$ by restricting the vertical axis to the interval $[0,50]$. }
\end{figure}

\begin{remark} (i) In Theorem \ref{thm2} we do not require neither $u$ nor $\Delta u$ to be positive in $B_R$. 

(ii) Classification of radially symmetric solutions for $\Delta ^m u=u^p$ may be found in \cite{DLS2014,LS2009a,LS2009b,LS2011}. 
\end{remark}

We shall next be interested on the behaviour at the boundary of solutions to \eq{eqtq} that satisfy either \eq{cond1} or  \eq{cond2}. Consider
\begin{equation}\label{eqtq1}
\left\{
\begin{aligned}
\Delta u&=v^{p},\, u>0&&\quad\mbox{ in }B_R,\\
\Delta v&=|\nabla u|^{q},\, v>0&&\quad\mbox{ in }B_R,\\
v&=\infty,&&\quad\mbox{ on }\partial B_R.\\
\end{aligned}
\right.
\end{equation}
Note first that according to Corollary \ref{corblowup3} the system \eq{eqtq1} has radially symmetric  solutions if and only if $q>1/p$.

Our main result regarding the behaviour of the radially symmetric solutions to \eq{eqtq1} is as follows.
\begin{theorem}\label{thm3}
Assume $p,q\geq 1$, $(p,q)\neq (1,1)$ and let $(u,v)$ be a positive radially symmetric solution to  \eq{eqtq1}. Then
\begin{equation}\label{blowv}
\lim_{|x|\nearrow R}(R-|x|)^{\frac{q+2}{pq-1}}v(x)=\Big[\frac{(1+2p)^{q}(q+2)(q+pq+1)}{(pq-1)^{2+q}}\Big]^{\frac{1}{pq-1}}.
\end{equation}
Also,
\begin{enumerate}
\item[(i)] If $q>2(1+1/p)$, then there exists $L:=\lim_{|x|\nearrow R}u(x)\in (0,\infty)$ and
\begin{equation}\label{blowu0}
\lim_{|x|\nearrow R}\frac{L-u(x)}{(R-|x|)^{\frac{pq-2(1+p)}{pq-1}}}=\Big[\frac{(1+2p)(q+2)^{p}(q+pq+1)^{p}}{(pq-1)^{2p+1}}\Big]^{\frac{1}{pq-1}}.
\end{equation}
\item[(ii)]If $q= 2(1+1/p)$ then
\begin{equation}\label{blowu2}
\lim_{|x|\nearrow R} \frac{u(x)}{\ln\frac{1}{R-|x|}}= \Big[\frac{(1+2p)(q+2)^{p}(q+pq+1)^{p}}{(pq-1)^{2p+1}}\Big]^{\frac{1}{pq-1}}.
\end{equation}
\item[(iii)]  If $q< 2(1+1/p)$ then 
\begin{equation}\label{blowu1}
\lim_{|x|\nearrow R}(R-|x|)^{\frac{2+2p-pq}{pq-1}}u(x)= \frac{pq-1}{2+2p-pq}\Big[\frac{(1+2p)(q+2)^{p}(q+pq+1)^{p}}{(pq-1)^{2p+1}}\Big]^{\frac{1}{pq-1}}.
\end{equation}
\end{enumerate}
\end{theorem}

Now we shall be interested in the system \eq{sys0} posed in the whole $\R^N$, namely 
\begin{equation}\label{sys1}
\left\{
\begin{aligned}
\Delta u&=v^p&&\quad\mbox{ in } \R^N,\\
\Delta v&=f(|\nabla u|) &&\quad\mbox{ in } \R^N.
\end{aligned}
\right.
\end{equation}

Our main result in this case is as follows.
\begin{theorem}\label{thmrn}
We have:
\begin{enumerate}
\item[(i)] The system \eq{sys1} has positive radially symmetric solutions if and only if
$$
\int_{1}^\infty\frac{ds}{\Big(\displaystyle \int_0^s F(t)dt  \Big)^{p/(2p+1)}} =\infty.
$$
\item[(ii)] Assume $f(t)=t^{q}$, where $q\geq 1>p$ and $pq<1$. Let $(u,v)$ be a positive radially symmetric solution. If 
\begin{equation}\label{div}
\frac{p(q^2-4)}{1-pq}\leq 2(N-1)
\end{equation}
then
\begin{equation*}
\lim_{|x|\rightarrow \infty}\frac{u(x)}{|x|^{\frac{2+2p-pq}{1-pq}}}=\Big[\frac{(2+q)(N(1-pq)+p(2+q))^{1/p}(N(1-pq)+q(2p+1))}{(1-pq)^{\frac{2p+2-pq}{p}}(2+2p-pq)^{\frac{pq-1}{p}}}\Big]^{\frac{p}{pq-1}}
\end{equation*}
and
\begin{equation*}
\lim_{|x|\rightarrow \infty}\frac{v(x)}{|x|^{\frac{q+2}{1-pq}}}=\Big[\frac{(2+q)(N(1-pq)+p(2+q))^{q}(N(1-pq)+q(2p+1))}{(1-pq)^{2+q}}\Big]^{\frac{1}{pq-1}}.
\end{equation*}

\end{enumerate}

\end{theorem}

\begin{remark} (i) Condition \eq{div} hold in particular if $q\leq 2$.

(ii) It is easy to see that $(u_0,v_0)$ given by
$$
\left\{
\begin{aligned}
u_{0}(x)&=\Big[\frac{(2+q)(N(1-pq)+p(2+q))^{1/p}(N(1-pq)+q(2p+1))}{(1-pq)^{\frac{2p+2-pq}{p}}(2+2p-pq)^{\frac{pq-1}{p}}}\Big]^{\frac{p}{pq-1}}  |x|^{\frac{2+2p-pq}{1-pq}}\\
 v_{0}(x)&=\Big[\frac{(2+q)(N(1-pq)+p(2+q))^{q}(N(1-pq)+q(2p+1))}{(1-pq)^{2+q}}\Big]^{\frac{1}{pq-1}}|x|^{\frac{2+q}{1-pq}}
\end{aligned}
\right.
$$
is a solution of \eq{sys1} with $f(t)=t^q$ that vanishes at the origin. Theorem \ref{thmrn} states  that any positive radial solution $(u,v)$ of \eq{sys1} with $f(t)=t^q$ behaves like $(u_{0},v_{0})$ at infinity.

The remaining of the paper is organised as follows. Section 3 contains a detour in dynamical systems; here we state the main tools which we use to study the asymptotic behaviour in Theorems \ref{thm3} and \ref{thmrn}. The following sections contain the proofs of our main results.

\end{remark}

\section{A detour in dynamical systems}
For any points $x=(x_1,x_2,x_3)$, $y=(y_1,y_2,y_3)$ in $\R^{3}$ define the open ordered interval
$$
[[x,y]]=\{z\in \R^3:x<z<y\}\subset \R^3.
$$
Consider the intial value problem
\begin{equation}\label{det1}
\left\{
\begin{aligned}
&\zeta_{t}=g(\zeta) \quad\mbox{ for } t\in \R,\\
&\zeta(0)=\zeta_{0},
\end{aligned}
\right.
\end{equation}
where $ g:\R^{3}\rightarrow \R$ is a $C^{1}$ function. This implies that for any $\zeta_{0}\in \R^{3},$ there exist a unique solution $\zeta$ of \eq{det1} defined in a maximal time interval. We denote by $\phi(\cdot,\zeta_{0})$ the flow associated to \eq{det1}, that is,  $t\longmapsto \phi(t,\zeta_{0})$ is the unique solution of \eq{det1} defined in maximal time interval. We shall assume that the vector field $g$ is cooperative, that is
$$
\frac{\partial g_{i}}{\partial x_{j}}\geq 0 \quad \mbox{ for } 1\leq i,j\leq 3,\;\; i\neq j.
$$
 The following results are due to Hirsch \cite{Hirsch1989, Hirsch1990}.
\begin{theorem}\label{thmdet1}{\rm (see \cite[Theorem1]{Hirsch1990})}
 Any compact limit set of \eq{det1} contains an equilibrium or is a cycle.
\end{theorem}
\begin{de}
A circuit is a finite sequence of equilibria $\zeta_{1},\zeta_{2},\dots,\zeta_{n}=\zeta_{1}$, $(n\geq 2)$ such that $W^{u}(\zeta_{i})\cap W^{s}(\zeta_{i+1})$ is non-empty, where $W^{u}$,\;\;$W^{s}$ denote the stable and unstable manifolds.
\end{de}
\begin{remark}
If all equilibria are hyperbolic and their stable and unstable manifolds are mutually transverse, then there cannot be any circuit.
\end{remark}

\begin{theorem}\label{thmdet2}{\rm (see \cite[Theorem 2]{Hirsch1990}). }
Let $K\subset \R^{3}$ be a compact set such that:
\begin{enumerate}
\item [ (i) ] All equilibria in $K$ are hyperbolic and there are no circuits.
\item [ (ii)] For any $T>0,$ the number of cycles in $K$ having period less than or equal to $T$ is finite.
\end{enumerate}
Then:
\begin{enumerate}
\item [ (a) ] Every limit set in $K$ is an equilibrium or cycle.
\item [ (b) ] The number of cycles in $K$ is finite.
\end{enumerate}

\end{theorem}

\begin{theorem}\label{thmdet3}{\rm (see \cite[Theorem 7]{Hirsch1989})}
Let $\zeta_{1}$, $\zeta_{2}\in \R^{3}$ such that $\zeta_1< \zeta_2$. If
$$
{\rm div}{g}<0 \quad\mbox{ in } [[\zeta_1,\zeta_2]],
$$ 
then \eq{det1} has no cycles in $[[\zeta_1,\zeta_2]].$
\end{theorem}

\begin{de}
A subset $A\subset \R^{3}$ is said to be positively invariant for the flow $\phi$ if $\phi(t,\zeta)\in A$ for all $\zeta \in A$ and $t\geq 0.$ $A$ is called invariant for $\phi$ if 
$$
\phi(t,A)= A \quad\mbox{ for all } t\geq 0,
$$
that is, for any $z\in A$ and $t\geq 0$ there exists $\zeta \in A$ such that $\phi(t,\zeta)=z.$

\medskip

The following notion of chain recurrence is due to Conley \cite{C1972,C1978}.
\end{de}

\begin{de}
Let $A\subset \R^{3}$ be a nonempty positively invariant subset for $\phi$ and $\zeta,\zeta'\in A.$
\begin{enumerate}
\item [(i)] For $\varepsilon> 0$ and $t> 0$, an $(\varepsilon,t)$-chain from $\zeta \in A$ to $\zeta'\in A$ is a sequence of points in $A$, $\zeta=\zeta_{1},\zeta_{2},\dots,\zeta_{n},\zeta_{n+1}=\zeta'$ and of times $t_{1},t_{2},\dots, t_{n}\geq t$ such that $|\phi(t_{i},\zeta_{i})-\zeta_{i+1}|< \varepsilon.$
\item [(ii)] A point $\zeta \in A$ is called chain recurrent if for every $\varepsilon> 0, t>0$ there is an $(\varepsilon,t)$-chain from $\zeta$ to $\zeta$ in $A.$
\item [(iii)] The set $A$ is said to be chain recurrent if every point $\zeta\in A$ is chain recurrent in $A.$
\end{enumerate}
\end{de}

\medskip

\begin{theorem}\label{thmmst1}{\rm (see \cite[Theorem 3.3]{C1972}, \cite[Lemma 1.4]{MST1995})}
If $A$ is connected and chain recurrent, then for any $\zeta,\zeta' \in A$ and any $\varepsilon, t>0$, there exists an $(\varepsilon,t)$ chain from $\zeta$ to $\zeta'.$
\end{theorem}

\bigskip

Consider now the initial value problem
\begin{equation}\label{fd1}
\left\{
\begin{aligned}
&\xi_{t}=G(t,\xi) \quad\mbox{ for } t>0,\\
&\xi(0)=\xi_{0}\in \R^{3},
\end{aligned}
\right.
\end{equation}
where $G$ is a $C^{1}$ function on $(0,\infty) \times \R^{3}$ such that $G(t,\cdot)\rightarrow g$ uniformly on compact subsets of $\R^{3}.$ We shall say that \eq{fd1} is asymptotically autonomous with the limit problem \eq{det1}. We denote by $\Phi(\cdot,\xi_{0})$ the semiflow defined by the initial value problem \eq{fd1}.

\medskip

The following result will play a crucial role in the proof of Theorem \ref{thm3}.

\begin{theorem}\label{thmfd2}{\rm (see \cite[Theorem 1.8]{MST1995})}
Let $\xi_{0}\in \R^{3}$ and assume the trajectory $\Phi(t,\xi_{0})$ associated to \eq{fd1} is bounded. Then the $\omega$-limit set $\omega_{\Phi}(\xi_{0})$ has the following properties:
\begin{enumerate}
\item [(a)] $\omega_{\Phi}(\xi_{0})$ is nonempty, compact and connected.
\item [(b)] $\omega_{\Phi}$ is invariant under the flow of $\phi$ of \eq{det1}, that is 
$$
\phi(t,\omega_{\Phi}(\xi_{0}))= \omega_{\Phi}(\xi_{0}) \quad\mbox{ for all } t\geq 0.
$$
\item [(c)] $\omega_{\Phi}(\xi_{0})$ attracts $\Phi(t,\xi_{0})$, that is,
$$
dist(\Phi(t,\xi_{0}),\omega_{\Phi}(\xi_{0}))\rightarrow 0 \quad\mbox{ as }t\rightarrow \infty.
$$
\item [(d)] $\omega_{\Phi}(\xi_{0})$ is chain recurrent for $\phi.$
\end{enumerate}
\end{theorem}

\medskip

In particular, Theorem \ref{thmfd2}{\rm (d)} states that the invariant set consisting of two equilibria $e_{1}, e_{2}$  and a heteroclinic orbit that joins them, or a homoclinic orbit connecting an equilibrium point $e_3$ with itself cannot be the $\omega$-limit set of an asymtotically autonomous semiflow.

\section{Proof of Theorem \ref{thm1} }
A useful result in proving Theorem \ref{thm1} is the following lemma.
\begin{lemma}\label{lemma1.1}
We have
\begin{equation}\label{equiv_cond}
\int_1^\infty\frac{ds}{\Big(\displaystyle \int_0^s F(t)dt  \Big)^{p/(2p+1)}} <\infty\mbox{ if and only if }
\int_1^\infty\frac{ds}{\Big(\displaystyle \int_0^s \sqrt{f(t)}dt  \Big)^{2p/(2p+1)}} <\infty. 
\end{equation}
Moreover,
$$
\int_1^\infty\frac{ds}{\Big(\displaystyle \int_0^{2s} {F(t)}dt  \Big)^{p/(2p+1)}} 
\leq \int_1^\infty\frac{ds}{\Big(\displaystyle \int_0^{s}\sqrt{f(t)}dt  \Big)^{2p/(2p+1)}} \leq \int_1^\infty\frac{ds}{\Big(\displaystyle \int_0^{2s} {F(t)}dt  \Big)^{p/(2p+1)}}.
$$
\end{lemma}
\begin{proof}
Let $H:[0,\infty)\to [0, \infty)$ be given by 
$$
H(s)=2\int_0^s F(t)dt-\Big(\int_0^{s}\sqrt{f(t)}{dt}\Big)^2\quad\mbox{ for all }s\geq 0.
$$
Then $$H'(s)=2F(s)-2\Big(\int_0^{s}\sqrt{f(t)}{dt}\Big)\sqrt{f(s)}\quad\mbox{ for all }s\geq 0.$$
Therefore 
\begin{equation}
\begin{aligned}
\frac{H'(s)}{2}&=F(s)-\sqrt{f(s)}\int_0^{s}\sqrt{f(t)}{dt}=\int_0^{s}f(t){dt}-\sqrt{f(s)}\int_0^{s}\sqrt{f(t)}{dt}\\
&=\int_0^{s}\Big(f(t)-\sqrt{f(s)}\sqrt{f(t)}{dt}=\int_0^{s}\sqrt{f(t)}\Big(\sqrt{f(t)}-\sqrt{f(s)}\Big){dt}\\
&\leq0 \quad\mbox{ for all }s\geq 0.
\end{aligned}
\end{equation}
Hence, $H$ is nonincreasing which yields 
$H(s)\leq H(0)=0$. This further implies 
$$2\int_0^{s}F(t){dt}\leq\Big(\int_0^{s}\sqrt{f(t)}{dt}\Big)^{2} \quad\mbox{ for all } s\geq 0.$$
So, $$\int_1^\infty\frac{ds}{\Big(2\int_0^{s}F(t){dt}\Big)^{p/(2p+1)}}\geq \int_1^\infty\frac{ds}{\Big(\int_0^{s}\sqrt{f(t)}{dt}\Big)^{2p/(2p+1)}}. $$
In order to establish the first inequality in our lemma, let $h:[0,\infty)\to \R$ be defined by 
$$
h(s)=\int_0^{2s}{F(t)}{dt}-\Big(\int_0^{s}\sqrt{f(t)}{dt}\Big)^2 \quad\mbox{ for all }s\geq 0.
$$
Then,
\begin{equation*}
\begin{aligned}
 \frac{h'(s)}{2}&=F(2s)-\sqrt{f(s)}\int_0^{s}\sqrt{f(t)}{dt}\\
&\geq\int_s^{2s}f(t){dt}-\sqrt{f(s)}\int_0^{s}\sqrt{f(t)}{dt}\\
&=\int_0^{s}f(s+t){dt}-\sqrt{f(s)}\int_0^{s}\sqrt{f(t)}{dt}.\\
&=\int_0^{s}\Big(f(t+s)-\sqrt{f(s)}\sqrt{f(t)}\Big){dt}\\
& \geq 0 \quad\mbox{ for all }s\geq 0.
\end{aligned}
\end{equation*}
It follows that $h$ is nondecreasing which yields $h(s)\geq h(0)=0$ for all $s\geq 0$. This implies
\begin{equation}\label{eqg}
\Big(\int_0^{2s}F(t){dt}\Big)^{p/(2p+1)}\geq\Big(\int_0^{s}\sqrt{f(t)}{dt}\Big)^{2p/(2p+1)} \quad\mbox{ for all }s\geq 0.
\end{equation}
Therefore $$\int_1^\infty\frac{ds}{\Big(\int_0^{2s}F(t){dt}\Big)^{p/(2p+1)}}\leq\int_1^\infty\frac{ds}{\Big(\int_0^{s}\sqrt{f(t)}{dt}\Big)^{2p/(2p+1)}}.$$
This concludes the proof.
\end{proof}

\noindent{\bf Proof of Theorem \ref{thm1}.}  It is enough to prove (ii) and (iii). We shall divide our proof into three steps.

\noindent{\it Step 1: Let $(u,v)$ be any positive radial solution of \eq{sys}. Then, letting $w=u'$ we have
\begin{equation}\label{eqw3}
 \frac{v^{p}(r)}{N}\leq w'(r)\leq v^{p}(r)\quad\mbox{ for all }0<r<R,
\end{equation}
and}
\begin{equation}\label{eqv3}
 \frac{f(w(r))}{N}\leq v''(r)\leq f(w(r)) \quad\mbox{ for all }0<r<R.
\end{equation}

Indeed, we note first that $(w,v)$ satisfies
\begin{equation}\label{sys01}
\left\{
\begin{aligned}
&w'(r)+\frac{N-1}{r}w(r)=v^p(r)&&\quad\mbox{for all }0<r<R, \\
&v''(r)+\frac{N-1}{r}v'(r)=f(|w(r)|)&&\quad\mbox{for all }0<r<R,\\
&w(0)=v'(0)=0, v(0)>0.
\end{aligned}
\right.
\end{equation}
Integrating in the first equation of \eq{sys01} we find
\begin{equation}\label{eqw}
w(r)=r^{1-N}\int_0^r t^{N-1}v^p(t)dt\quad\mbox{ for all }0<r<R.
\end{equation}
This implies that $w>0$ in $(0,R)$ so $u$ is increasing. We now integrate in the second equation of \eq{sys01} to deduce
\begin{equation}\label{eqv1b}
v'(r)=r^{1-N}\int_0^r t^{N-1}f(w(t))dt\quad\mbox{ for all }0<r<R.
\end{equation}
This means $v'>0$ and $v$ is increasing on $(0,R)$. Using this fact in \eq{eqw} we find
\begin{equation}\label{eqw2}
w(r)\leq \frac{r}{N}v^p(r)\quad\mbox{ for all }0<r<R.
\end{equation}
Combining \eq{eqw2} with the first equation of \eq{sys01} we have
$$
v^{p}(r)\leq w'(r)+\frac{N-1}{N}v^{p}(r) \quad\mbox{ for all }0<r<R
$$
which implies \eq{eqw3}.
In particular $w'>0$ in $(0,R)$ so $w$ is increasing.
Using \eq{eqv1b} we obtain
\begin{equation}\label{eqv2}
 v'(r)\leq \frac{r}{N}f(w(r))\quad\mbox{ for all }0<r<R.
\end{equation}
Using \eq{eqv2} in the second equation of \eq{sys01} and also  the fact that $w>0$ we deduce the estimate \eq{eqv3}.
\medskip

\noindent {\it Step 2: System \eq{sys} admits a positive radial solution $(u,v)$ such that $\lim_{r\nearrow R} v(r)=\infty$ if and only if }
\begin{equation}\label{int01}
\int_{1}^\infty\frac{ds}{\Big(\displaystyle \int_0^s F(t)dt  \Big)^{p/(2p+1)}} <\infty.
\end{equation}
Assume first that $(u,v)$ is a positive radial solution of \eq{sys} with $\lim_{r\nearrow R} v(r)=\infty$.

Using \eq{eqv3}  we have 
$$
v''(r)\leq f(w(r)) \quad\mbox{ for all }0<r<R.
$$
Multiplying the above inequality by $v'(r)$ and then integrating over $[0,r]$  we have
\begin{equation*}
\begin{aligned}
\frac{(v'(r))^{2}}{2}&\leq \int_0^{r}v'(t)f(w(t)){dt}\\
&\leq f(w(r))\int_0^r v'(t)dt \\
&\leq f(w(r))v(r)  \quad\mbox{ for all }0<r<R,
\end{aligned}
\end{equation*}
which gives
$$
v'(r)(v(r))^{-1/2}\leq C\sqrt{f(w(r))} \quad\mbox{ for all }0<r<R, \mbox{ where } C>0.
$$
Multiplying the above inequality by $w'(r)$ and then using \eq{eqw3} we have
$$
v'(r)\frac{v^{p-1/2}(r)}{N}\leq w'(r)v'(r)(v(r))^{-1/2}\leq Cw'(r)\sqrt{f(w(r))}
$$
for all $0<r<R$. This further implies
$$
\frac{1}{N}\left(\frac{v^{p+1/2}(r)}{p+1/2}\right)'\leq Cw'(r)\sqrt{f(w(r))} \quad\mbox{ for all }0<r<R.
$$
Integrating over $[0,r]$ we obtain
$$
v^{p+1/2}(r)-v^{p+1/2}(0)\leq C\int_{w(0)=0}^{w(r)}\sqrt{f(t)}{dt}.
$$
Since $\lim_{r\nearrow R}v(r)=\infty$, we can find $\rho \in (0,R)$ such that
$$
\Big(v^{p}(r)\Big)^{(2p+1)/2p}\leq C\int_0^{w(r)}\sqrt{f(t)}{dt} \quad\mbox{ for all }\rho\leq r<R.
$$
Using \eq{eqw3} we obtain
$$
\frac{w'(r)}{\displaystyle\Big(\int_0^{w(r)}\sqrt{f(t)}{dt}\Big)^{2p/(2p+1)}}\leq C  \quad\mbox{ for all }\rho\leq r<R.
$$
Integrating the above inequality over $[\rho,r]$ we have
$$
\int_{\rho}^{r}\frac{w'(t)}{\displaystyle\Big(\int_0^{w(t)}\sqrt{f(t)}{dt}\Big)^{2p/(2p+1)}}{dt}\leq C(r-\rho)\leq Cr \quad\mbox{ for all }\rho\leq r<R.
$$
By changing the variable and then letting $r\nearrow R$  one  obtains
\begin{equation}\label{eqww}
\int_{w(\rho)}^{\infty}\frac{ds}{\displaystyle\Big(\int_0^{s}\sqrt{f(t)}{dt}\Big)^{2p/(2p+1)}}\leq C(R-\rho)< \infty.
\end{equation}
Hence,
$$
\int_{1}^{\infty}\frac{ds}{\displaystyle\Big(\int_0^{s}\sqrt{f(t)}{dt}\Big)^{2p/(2p+1)}} < \infty .
$$
Using Lemma \ref{lemma1.1}, this is equivalent to 
$$
\int_{1}^{\infty}\frac{ds}{\displaystyle\Big(\int_0^{s}{F(t)}{dt}\Big)^{p/(2p+1)}} < \infty.
$$

We now assume that $f$ fulfills \eq{int01} and prove that \eq{sys} has a positive radial solution $(u,v)$ satisfying 
$\lim_{r\nearrow R} v(r)=\infty$.
Looking for radially symmetric solutions of \eq{sys} we are led to solve
\begin{equation}\label{sysr}
\left\{
\begin{aligned}
&u''(r)+\frac{N-1}{r}u'(r)=v^p(r),\;r>0\\
&v''(r)+\frac{N-1}{r}v'(r)=f(|u'(r)|),\;r>0\\
&u(r)>0, v(r)>0 \mbox{ for } r\geq 0.
\end{aligned}
\right.
\end{equation}
In order to obtain the local existence of a solution, it is more convenient to introduce $w=u'$. Thus, the system \eq{sysr} reads  
\begin{equation}\label{sysr1}
\left\{
\begin{aligned}
&u'(r)=w(r)\\
&w'(r)+\frac{N-1}{r}w(r)=v^p(r)\\
&v''(r)+\frac{N-1}{r}v'(r)=f(|w(r)|)\\
&w(0)=v'(0)=0, v(0)>0.
\end{aligned}
\right.
\end{equation}
By twice integration, \eq{sysr1} is equivalent to
\begin{equation}\label{sysr2}
\left\{
\begin{aligned}
&u(r)=u(0)+\int_0^r w(t)dt,\;r>0,\\
&w(r)=\int_0^r t^{N-1}v^p(t)dt,\; r>0,\\
&v(r)=v(0)+\int_0^r t^{1-N}\int_0^t s^{N-1}f(|w(s)|)dsdt, \; r>0,\\
&u(0)>0, v(0)>0.
\end{aligned}
\right.
\end{equation}
Since $f$ is a $C^1$-function, by a standard contraction mapping principle one obtains the existence of a solution $(u,v)$ of \eq{sysr} defined in a maximal interval $[0,R_{max})$. By Step 1, $w$ and $v$ satisfy 
\eq{eqw3} and \eq{eqv3}. Thus,  we have
\begin{equation}\label{eqwv}
\left\{
\begin{aligned}
&f(w(r))\leq Nv''(r)  &&\quad\mbox{ for all }0<r<R_{max}, \\
& w'(r)\leq v^{p}(r)  &&\quad\mbox{ for all }0<r<R_{max}.
\end{aligned}
\right.
\end{equation}
Multiplying the two inequalities in \eq{eqwv} and then integrating over $[0,r]$  we deduce 
$$
F(w(r))\leq Nv^p(r)v'(r) \quad\mbox{ for all }0<r<R_{max}.
$$
Multiplying the above inequality by $w'(r)$ and using \eq{eqw3} one  obtains
\begin{equation}\label{eqwv2}
 w'(r)F(w(r))\leq Nv^{2p}(r)v'(r) \quad\mbox{ for all }0<r<R_{max}.
\end{equation}
Fix $\rho\in (0,R_{max})$ and  denote 
$$
G(r):=\int_{\rho}^{r}{F(t)}{dt} \quad\mbox{ for all }\rho\leq r<R_{max}.
$$
Integrating \eq{eqwv2} over $[\rho,r]$ and using \eq{eqw3} we have
\begin{equation*}
\begin{aligned}
G(w(r))=\int_{\rho}^{w(r)}{F(t)}{dt}&\leq N\int_{\rho}^{r}{v^{2p}(t)v'(t)}{dt}\\
&\leq C[v^p(r)]^{(2p+1)/p}\\
&\leq C[w'(r)]^{(2p+1)/p} \quad\mbox{ for all }\rho\leq r<R_{max}.
\end{aligned}
\end{equation*}
Hence
$$
C\leq \frac{w'(r)}{\Big(G(w(r))\Big)^{p/(2p+1)}}\quad\mbox{ for all }\rho\leq r<R_{max}.
$$
A further integration over $[\rho,r]$ yields
$$
C(r-\rho)\leq \int_{\rho}^{r}\frac{w'(t)}{\Big(G(w(t))\Big)^{p/(2p+1)}}{dt} \rho\leq r<R_{max}.
$$
By changing the variable of integration and then letting $r \nearrow R_{max}$  we have
\begin{equation}\label{eqG}
C(R_{max}-\rho)\leq \int_{w(\rho)}^{\infty}\frac{ds}{\Big(G(s)\Big)^{p/(2p+1)}}.
\end{equation}
Therefore,
$$
R_{max}\leq C\int_{1}^{\infty}\frac{ds}{\displaystyle\Big(\int_0^{s}{F(t)}{dt}\Big)^{p/(2p+1)}}< \infty.
$$
We have obtained a positive radial solution $(u,v)$ of \eq{sys} in $B_{R_{max}}$ satisfying $\lim_{r\nearrow R_{max}}v(r)=\infty$. Now, if $R>0$ is any positive radius, we let 
$$
\tilde f(t)= \lambda^{2(1+1/p)} f\Big(\frac{t}{\lambda}\Big)\quad \mbox{ for all } t\geq 0.
$$ 
Clearly $\tilde f$ satifies \eq{int01}. By the above arguments there exists $(\tilde u, \tilde v)$  such that
$$
\left\{
\begin{aligned}
\Delta \tilde u&=\tilde v^p&&\quad\mbox{ in }B_{R_{max}},\\
\Delta \tilde v&=\tilde f(|\nabla \tilde u|) &&\quad\mbox{ in }B_{R_{max}},
\end{aligned}
\right.
$$
where $B_{R_{max}}$ is a maximum ball of existence.
Let
\begin{equation*}
\left\{
\begin{aligned}
u(x)&=\tilde u\Big(\frac{x}{\lambda}\Big) &&\quad\mbox{ in } B_R,\\
v(x)&=\lambda^{-2/p}\tilde v\Big(\frac{x}{\lambda}\Big) &&\quad\mbox{ in } B_R.
\end{aligned}
\right.
\end{equation*}
By taking $\lambda =R/R_{max}$, we deduce that $(u,v)$ satisfies \eq{sys} in $B_R$.

\bigskip

\bigskip

\bigskip

\noindent{\it Step 3: Proof of (ii) and (iii).}

Assume \eq{sys} admits a positive radial solution $(u,v)$ in $B_R$ that satisfies \eq{cond1} (resp. \eq{cond2}). By Step 2 above, $f$ must satisfy \eq{int01}. From \eq{eqg} we have
$$
\Big(\int_0^{2s}F(t){dt}\Big)^{p/(2p+1)}\geq\Big(\int_0^{s}\sqrt{f(t)}{dt}\Big)^{2p/(2p+1)} \quad\mbox{ for all }s\geq 0.
$$
 Using this fact and working in the same way as we did for estimating \eq{eqww} and \eq{eqG}, there exists $\rho\in (0,R)$ such that
\begin{equation}\label{eqint}
\int_{w(r)}^{\infty}\frac{ds}{\displaystyle\Big(\int_0^{2s}F(t){dt}\Big)^{p/(2p+1)}}\leq \int_{w(r)}^{\infty}\frac{ds}{\displaystyle\Big(\int_0^{s}\sqrt{f(t)}{dt}\Big)^{2p/(2p+1)}}\leq C_1(R-r)
\end{equation}
and
\begin{equation}\label{eqint2}
\int_{w(r)}^{\infty}\frac{ds}{\displaystyle\Big(\int_0^{s}F(t){dt}\Big)^{p/(2p+1)}}\geq  C_2(R-r)
\end{equation}
for all $\rho<r<R$, where $C_1,C_2>0$ are constants.

Let $\Gamma: (0,\infty)\to (0,\infty)$ be defined as
$$
\Gamma{(t)}=\int_{t}^{\infty}\frac{ds}{\displaystyle\Big(\int_0^{s}{F(t)}{dt}\Big)^{p/(2p+1)}}\quad\mbox{ for all }t>0.
$$
Note that $\Gamma$ is decreasing and by \eq{int01} we have $\lim_{t\to \infty}\Gamma(t)=0$.
From \eq{eqint} and \eq{eqint2} we deduce
$$
\Gamma{(2w(r))}\leq C_1(R-r) \quad\mbox{ and }\quad
\Gamma{(w(r))}\geq C_2(R-r)\quad\mbox{ for all }\rho\leq r<R.
$$
Since $\Gamma$ is decreasing, the above estimates  yield
\begin{equation}\label{equ}
 \left\{
\begin{aligned}
2w(r)&\geq \Gamma^{-1}(C_1(R-r)) &&\quad\mbox{ for all }\rho\leq r<R,\\
w(r)&\leq \Gamma^{-1}(C_2(R-r)) &&\quad\mbox{ for all }\rho\leq r<R.
\end{aligned}
\right.
\end{equation}
Let us recall that
\begin{equation}\label{equmax}
u(r)=u{(\rho)}+\int_{\rho}^{r}{w(t)}{dt} \quad\mbox{ for all }\rho\leq r<R.
\end{equation}
From \eq{equmax} and \eq{equ} we find $\lim_{r\nearrow R}u(r)=\infty$ if and only if
$$
\int_{\rho}^{R}w(t){dt}=\infty
$$
if and only if
$$
\int_{\rho}^{R}{\Gamma^{-1}(C(R-t))}{dt}=\infty,
$$
for some constant $C>0$. Hence $\lim_{r\nearrow R}u(r)=\infty$ if and only if
$$
\int_0^{C(R-\rho)}{\Gamma^{-1}(\sigma)}{d\sigma}=\infty,
$$
if and only if
$$
\int_0^{1}{\Gamma^{-1}(\sigma)}{d\sigma}=\infty.
$$
With the change of variable $t=\Gamma^{-1}(\sigma)$ we now obtain $\lim_{r\nearrow R}u(r)=\infty$ if and only if
$$
\int_{1}^{\infty}\frac{s}{\Big(\int_0^{s}F(t){dt}\Big)^{p/(2p+1)}}{ds}=\infty.
$$
This implies \eq{int2}.

We now assume that \eq{int2} holds and show that system \eq{sys} has a positive radial solution $(u,v)$ that satisfies \eq{cond2}. We proceed as in Step 2. First we obtain the (local) existence of such a solution in a ball $B_{R_{max}}$ and then, by the same scaling argument indicated at the end of Step 2 we are able to conclude the existence of the desired solution to \eq{sys} in $B_R$ that satisfies \eq{cond2}.

\section{Proof of Theorem \ref{thm2}}  
Let $v=\Delta u$. Then $(u,v)$ satisfies
\begin{equation}\label{sysr3}
\left\{
\begin{aligned}
&u''(r)+\frac{N-1}{r}u'(r)=v(r) &&\quad\mbox{ for all }0<r<R,\\
&v''(r)+\frac{N-1}{r}v'(r)=f(|u'(r)|) &&\quad\mbox{ for all }0<r<R,\\
&\lim_{r\nearrow R}u(r)=\infty. 
\end{aligned}
\right.
\end{equation}
Integrating twice in the second equation of \eq{sysr3} we find that $r\longmapsto r^{N-1}v'(r)$ is increasing on $(0,R)$, which further implies $v'$ is increasing on $(0,R)$.
Thus, there exist
$$
L:=\lim_{r\nearrow R}v(r)\in \R\cup\{\infty\}.
$$
If $L<\infty$, then $v \leq L$  for all $0<r<R$. This fact combined with a twice integration in \eq{sysr3} yields 
$$
u(r)- u(0)\leq \frac{Lr^{2}}{2N}\quad\mbox{ for all }0<r<R.
$$
This implies that $u$ is bounded over $[0,R)$ which is a contradiction.
Therefore we must have $\lim_{r \nearrow R}v(r)=\infty$. In view of this fact there exists $R_{0}>0$ such that
$$
v(r)>0   \quad\mbox{ for all } R_{0}\leq r<R.
$$
Let $u'=w$. Then, from \eq{sysr3} we 
deduce
\begin{equation}\label{sysr5}
\left\{
\begin{aligned}
&w'(r)+\frac{N-1}{r}w(r)=v(r) &&\quad\mbox{ for all }0<r<R,\\
&v''(r)+\frac{N-1}{r}v'(r)=f(|w(r)|) &&\quad\mbox{ for all }0<r<R.
\end{aligned}
\right.
\end{equation}
From the first equation of \eq{sysr5} we find 
\begin{equation}\label{eqa1}
\Big(r^{N-1}w\Big)'=r^{N-1}v\geq 0 \quad\mbox{ for all } R_{0}\leq r<R,
\end{equation}
which in particular implies that $r\longmapsto r^{N-1}w$ is increasing on $[R_0,R)$ and so, there exists 
$$
L_{0}:=\lim_{r\nearrow R}w(r)=\lim_{r\nearrow R}u'(r)\in \R\cup \{\infty\}.
$$
With a similar argument as above, if $L_0$ is finite we derive that $u$ is bounded which contradicts  $\lim_{r\nearrow R}u(r)=\infty$. Hence   $\lim_{r\nearrow R}w(r)=\infty$.
Integrating \eq{eqa1} over $[R_{0},r]$ we obtain
$$
r^{N-1}w(r)-R_{0}^{N-1}w(R_{0})=\int_{R_0}^r t^{N-1}v(t)dt\leq v(r)\int_{R_0}^r t^{N-1}dt =\frac{v(r)}{N}\Big(r^N-R_{0}^{N}\Big),
$$
for all $R_0<r<R$. This yields
\begin{equation}\label{eqa2}
w(r)\leq \frac{v(r)r}{N}+\frac{C}{r^{N-1}} \quad\mbox{ for all }R_0<r<R,
\end{equation}
where $C=C(R_{0},N)>0$ is a constant. Since $v(r)\to \infty$ as $r\nearrow R$, we may choose $R_{1}\in (R_{0},R)$ such that 
\begin{equation}\label{eqa3}
\frac{C}{R_{1}^{N-1}}\leq \frac{v(r)r}{2N(N-1)} \quad\mbox{ for all }R_1<r<R.
\end{equation}
Combining \eq{eqa2} and \eq{eqa3} we obtain
\begin{equation}\label{eqa4}
w(r)\leq \frac{2N-1}{2N(N-1)} v(r)r \quad\mbox{ for all }R_1<r<R.
\end{equation}
We now use this last estimate in the first equation of \eq{sysr5} to deduce
$$
\frac{v(r)}{2N} \leq w'(r)\leq v(r) \quad\mbox{ for all } R_1<r<R.
$$
The same approach is now applied to the second equation of \eq{sysr5} in order to deduce
$$
\frac{f(w(r))}{2N}\leq v''(r) \leq f(w(r)) \quad\mbox{ for all } R_1<r<R.
$$
From now on, we follow line by line  the proof of Theorem \ref{thm1} with $p=1$ to reach the required conclusion.

\section{Proof of Theorem \ref{thm3}}

\subsection{More properties of solutions to system \eq{eqtq1}}

Let $(u,v)$ be a positive radially symmetric solution of \eq{eqtq1} in $B_R$.

Letting $u'(r)=w(r)$ and $v'(r)= \psi(r)$ we have 
\begin{equation}\label{ds3}
\left\{
\begin{aligned}
&w'(r)+\frac{N-1}{r}w(r)= v^{p}(r) &&\quad\mbox{ for } 0<r<R, \\
&v'(r)= \psi(r) &&\quad\mbox{ for } 0<r<R, \\
&\psi'(r)+\frac{N-1}{r}\psi(r)= w^{q}(r) &&\quad\mbox{ for } 0<r<R,\\
&w(0)=0, v(0)=m>0,\psi(0)=0. 
\end{aligned}
\right.
\end{equation}

The next result is a comparison principle between sub and supersolutions of \eq{ds3} which is true in virtue of the quasimonotone character of our system.

\begin{lemma}\label{lemma1.2}
Let $(v_{1}(r),w_{1}(r),\psi_{1}(r))$ and $(v_{2}(r),w_{2}(r),\psi_{2}(r))$ be solutions of
\begin{equation}\label{com1}
 \left\{
\begin{aligned}
&w_{1}'(r)+\frac{\theta}{r}w_{1}(r)\geq v_{1}^{p}(r) &&\quad\mbox{ for } 0<r<R_{1}^{max}, \\
&v_{1}'(r)\geq \psi_{1}(r) &&\quad\mbox{ for } 0<r<R_{1}^{max}, \\
&\psi_{1}'(r)+\frac{\theta}{r}\psi_{1}(r)\geq w_{1}^{q}(r) &&\quad\mbox{ for } 0<r<R_{1}^{max},\\
&w_{1}(0)=\mu_{1}, v_{1}(0)=m_{1},\psi_{1}(0)=\nu_{1}. 
\end{aligned}
\right.
\end{equation}

\begin{equation}\label{com2}
 \left\{
\begin{aligned}
&w_{2}'(r)+\frac{\theta}{r}w_{2}(r)\leq v_{2}^{p}(r) &&\quad\mbox{ for } 0<r<R_{2}^{max}, \\
&v_{2}'(r)\leq \psi_{2}(r) &&\quad\mbox{ for } 0<r<R_{2}^{max}, \\
&\psi_{2}'(r)+\frac{\theta}{r}\psi_{2}(r)\leq w_{2}^{q}(r) &&\quad\mbox{ for } 0<r<R_{2}^{max},\\ 
&w_{2}(0)=\mu_{2}, v_{2}(0)=m_{2},\psi_{2}(0)=\nu_{2}.
\end{aligned}
\right.
\end{equation}
where $\theta \geq 0$,\;\; $\mu_{i}\geq 0$,\;\; $\nu_{i}\geq 0$,\;\;$m_{i}\geq 0$ ,\;\;$(i=1,2).$ 

If 
$$
\mu_{1}\geq \mu_{2},\;\;m_{1}\geq m_{2},\;\;\nu_{1}\geq \nu_{2},\;\; (\mu_{1},m_{1},\nu_{1})\neq (\mu_{2},m_{2},\nu_{2}).
$$
Then 
$$
w_{1}(r)>w_{2}(r),\;\; v_{1}(r)>v_{2}(r),\;\;\psi_{1}(r)>\psi_{2}(r) \quad\mbox{ for all } 0<r<min\{R_{1}^{max},R_{2}^{max}\}.
$$
\end{lemma}
\begin{proof}
Without loosing any generality we may assume that $m_{1}>m_{2}.$ Note that the first equation in \eq{com1} and \eq{com2} can be written as
\begin{equation}\label{com3}
\left\{
\begin{aligned}
&\Big(r^{\theta}w_{1}(r)\Big)'\geq r^{\theta}v_{1}^{p}(r) &&\quad\mbox{  for }0<r<R_{1}^{max}, \\
&\Big(r^{\theta}w_{2}(r)\Big)'\leq r^{\theta}v_{2}^{p}(r) &&\quad\mbox{  for }0<r<R_{2}^{max}.
\end{aligned}
\right.
\end{equation}
Since $v_{1}(0)=m_{1}>v_{2}(0)=m_{2}$, there exist $\rho>0$ such that
\begin{equation}\label{com4}
v_{1}(r)> v_{2}(r) \quad\mbox{ for all } 0\leq r\leq \rho.
\end{equation}
Using \eq{com3} and \eq{com4} we have
$$
\Big(r^{\theta}w_{1}(r)\Big)'\geq r^{\theta}v_{1}^{p}(r) > r^{\theta}v_{2}^{p}(r) \geq \Big(r^{\theta}w_{2}(r)\Big)'\quad\mbox{ for all } 0\leq r\leq \rho.
$$
Hence
\begin{equation}\label{com5}
 \Big(r^{\theta}w_{1}(r)\Big)'> \Big(r^{\theta}w_{2}(r)\Big)' \quad\mbox{ for all } 0\leq r\leq \rho.
\end{equation}
Integrating \eq{com5} over $[0,r]$, $r\leq \rho$, we find
\begin{equation}\label{com6}
w_{1}(r)> w_{2}(r) \quad\mbox{ for all } 0\leq r\leq \rho.
\end{equation}
Using \eq{com6} in the third equation of \eq{com1} and \eq{com2} we get 
$$
\Big(r^{\theta}\psi_{1}(r)\Big)'> \Big(r^{\theta}\psi_{2}(r)\Big)' \quad\mbox{ for all } 0\leq r\leq \rho.
$$
Integrating over $[0,r]$, $r\leq \rho$, we deduce that
$$
\psi_{1}(r)> \psi_{2}(r)  \quad\mbox{ for all } 0\leq r\leq \rho.
$$
Let us denote
\begin{equation}\label{com7}
R:=\sup\Big\{\eta> 0 : v_{1}(r)>v_{2}(r), w_{1}(r)>w_{2}(r), \psi_{1}(r)> \psi_{2}(r) \mbox{ in }(0,\eta)\Big\}.
\end{equation}
First of all note that $R>0.$ Using the same arguments as above we have 
$$
R=\min\{R_{1}^{max}, R_{1}^{max}\}.
$$
\end{proof}

\begin{lemma}\label{lemma1.3}
Let $(v_{1}(r),w_{1}(r),\psi_{1}(r))$ and $(v_{2}(r),w_{2}(r),\psi_{2}(r))$ be the solutions of
\begin{equation}\label{com8}
 \left\{
\begin{aligned}
&w_{1}'(r)+\frac{N-1}{r}w_{1}(r)=v_{1}^{p}(r) &&\quad\mbox{ for } 0<r<R_{1}^{max}, \\
&v_{1}'(r)=\psi_{1}(r) &&\quad\mbox{ for } 0<r<R_{1}^{max}, \\
&\psi_{1}'(r)+\frac{N-1}{r}\psi_{1}(r)=w_{1}^{q}(r) &&\quad\mbox{ for } 0<r<R_{1}^{max},\\
&v_{1}(R_{1}^{max})=\infty,\\
&w_{1}(0)=0, v_{1}(0)=m_{1}>0,\psi_{1}(0)=0
\end{aligned}
\right.
\end{equation}
and
\begin{equation}\label{com9}
 \left\{
\begin{aligned}
&w_{2}'(r)+\frac{N-1}{r}w_{2}(r)=v_{2}^{p}(r) &&\quad\mbox{ for } 0<r<R_{2}^{max}, \\
&v_{2}'(r)=\psi_{2}(r) &&\quad\mbox{ for } 0<r<R_{2}^{max}, \\
&\psi_{2}'(r)+\frac{N-1}{r}\psi_{2}(r)=w_{2}^{q}(r) &&\quad\mbox{ for } 0<r<R_{2}^{max},\\ 
&v_{2}(R_{2}^{max})=\infty,\\
&w_{2}(0)=0, v_{2}(0)=m_{2}>0,\psi_{2}(0)=0.
\end{aligned}
\right.
\end{equation}
If $m_{1}>m_{2}$ then $R_{1}^{max}< R_{2}^{max}$.
\end{lemma}
\begin{proof}
Let $\sigma> 0$ be such that 
\begin{equation}\label{com10}
 \Big(\frac{m_{2}}{m_{1}}\Big)^{\frac{pq-1}{q+2}}< \sigma< 1.
\end{equation}
Since $m_{1}> m_{2}$ and $pq-1> 0$, it is always possible to choose such a $\sigma.$ Then
$$
\tilde{w}(r)=\sigma^\frac{2p+1}{pq-1}w_{1}(\sigma r),\;\;\tilde{v}(r)=\sigma^\frac{q+2}{pq-1}v_{1}(\sigma r),\;\; \tilde{\psi}(r)=\sigma^\frac{pq+q+1}{pq-1}\psi_{1}(\sigma r)
$$
satisfies
\begin{equation}\label{com12}
\left\{
\begin{aligned}
&\tilde{w}'(r)+\frac{N-1}{r}\tilde{w}(r)=\tilde{v}^{p}(r)\\
&\tilde{v}'(r)=\tilde{\psi}(r)\\
&\tilde{\psi}'(r)+\frac{N-1}{r}\tilde{\psi}(r)=\tilde{w}^{q}(r).\\
&\tilde{w}(0)=0,\;\;\tilde{\psi}(0)=0,\;\;\tilde{v}(0)=\sigma^\frac{q+2}{pq-1}m_{1}> m_{2}.
\end{aligned}
\right.
\end{equation}

Using Lemma \ref{lemma1.2} for $(\tilde{v}(r),\tilde{w}(r),\tilde{\psi}(r))$ and $(v_{2}(r),w_{2}(r),\psi_{2}(r))$ it follows that
$$
\tilde{v}(r)> v_{2}(r),\;\; \tilde{w}(r)> w_{2}(r),\;\; \tilde{\psi}(r)> \psi_{2}(r).
$$
In particular $\tilde{v}(r)=\sigma^\frac{q+2}{pq-1}v_{1}(\sigma r)> v_{2}(r).$ Since $v_{1}(\sigma r)$ is finite as long as $\sigma r< R_{1}^{max}$, it follows that $v_{2}(r)$ is finite as long as $r< \frac{R_{1}^{max}}{\sigma}.$ Since $R_{2}^{max}$ is the maximum time of existence for $v_{2}(r)$, it follows 
$$
\frac{R_{1}^{max}}{\sigma}\leq R_{2}^{max}.
$$
Using the fact that $\sigma< 1,$ we have
$$
R_{1}^{max}\leq \sigma R_{2}^{max}< R_{2}^{max},
$$ 
which completes the proof.
\end{proof}

\subsection{Proof of Theorem \ref{thm3}} We may always assume $R=1$ because if $(u,v)$ is a solution of \eq{eqtq1} in $B_{R}$ then
\begin{equation}\label{ds1}
\left\{
\begin{aligned}
&U(x)= R^{\frac{2+2p-pq}{pq-1}}u(Rx) &&\quad\mbox{ for } x\in B_{1},\\
&V(x)= R^{\frac{2+q}{pq-1}}v(Rx)  &&\quad\mbox{ for } x\in B_{1},
\end{aligned}
\right.
\end{equation}
is a solution of \eq{eqtq1} in $B_{1}$. In the sequel we shall assume $R=1$. Thus, letting $w=u'$ and $\psi=v'$ we have that $(v,w,\psi)$ satisfies 

\begin{equation}\label{ds2}
 \left\{
\begin{aligned}
&w'(r)+\frac{N-1}{r}w(r)=v^{p}(r) &&\quad\mbox{ for } 0<r<1, \\
&v'(r)=\psi(r) &&\quad\mbox{ for } 0<r<1, \\
&\psi'(r)+\frac{N-1}{r}\psi(r)=w^{q}(r) &&\quad\mbox{ for } 0<r<1,\\
&w(0)=0, v(0)=m>0,\psi(0)=0,\\
&\lim_{r\nearrow 1}w(r)= \lim_{r\nearrow 1} v(r)=\lim_{r\nearrow 1} \psi(r)= \infty.
\end{aligned}
\right.
\end{equation}

Further, let
\begin{equation}\label{eqhet1}
a(r)=\frac{w(r)}{A}(1-r)^{\alpha},\;\;
b(r)=\frac{v(r)}{B}(1-r)^{\beta},\;\;
c(r)=\frac{\psi(r)}{C}(1-r)^{\gamma},
\end{equation}
where
\begin{equation}\label{dsa1}
\begin{aligned}
A&=\Big[\frac{(1+2p)(q+2)^{p}(q+pq+1)^{p}}{(pq-1)^{2p+1}}\Big]^{\frac{1}{pq-1}},\\
B&=\Big[\frac{(1+2p)^{q}(q+2)(q+pq+1)}{(pq-1)^{2+q}}\Big]^{\frac{1}{pq-1}},\\ 
C&=\Big[\frac{(1+2p)^{q}(q+2)^{pq}(q+pq+1)}{(pq-1)^{q+pq+1}}\Big]^{\frac{1}{pq-1}},
\end{aligned}
\end{equation}
and
\begin{equation}\label{dsb2}
\alpha=\frac{1+2p}{pq-1},\;\;\beta=\frac{q+2}{pq-1} ,\;\;\gamma=\frac{q+pq+1}{pq-1}.
\end{equation}
 From \eq{ds2} we deduce that $(a,b,c)$ satisfies
\begin{equation}\label{ds4}
 \left\{
\begin{aligned}
 &(1-r)\Big(a'(r)+\frac{N-1}{r}a(r)\Big)= \alpha(b^{p}(r)-a(r)) &&\quad\mbox{ for } 0<r<1, \\
&(1-r)b'(r)+\beta b(r)= \beta c(r) &&\quad\mbox{ for } 0<r<1, \\
&(1-r)\Big(c'(r)+\frac{N-1}{r}c(r)\Big)= \gamma(a^{q}(r)-c(r)) &&\quad\mbox{ for } 0<r<1,\\
&a(0)=0, b(0)=m/B>0, c(0)=0.
\end{aligned}
\right.
\end{equation}
We next introduce a new change of variable in the system \eq{ds4}, by letting $r=1-e^{-t}$ and $X(t)=a(r)$, $Y(t)=b(r)$ and $Z(t)=c(r)$ where $t=\ln(\frac{1}{1-r})$. Thus, \eq{ds4} yields 
\begin{equation}\label{ds5}
 \left\{
\begin{aligned}
&X_{t}+\frac{N-1}{e^{t}-1}X(t)=\alpha(Y^{p}(t)-X(t)) &&\quad\mbox{ for } 0<t<\infty, \\
&Y_{t}=\beta(Z(t)-Y(t)) &&\quad\mbox{ for } 0<t<\infty, \\
&Z_{t}+\frac{N-1}{e^{t}-1}Z(t)=\gamma(X^{q}(t)-Z(t)) &&\quad\mbox{ for } 0<t<\infty,\\
&X(0)=0, Y(0)=m/B>0, Z(0)=0.
\end{aligned}
\right.
\end{equation}

\bigskip

The proof of Theorem \ref{thm3} will be divided into three steps. 

\medskip

\noindent {\it Step 1: $\xi(t)=(X(t),Y(t),Z(t))$ is bounded as $t\rightarrow \infty$. }

\medskip
Let us assume by contradiction that $(X(t),Y(t),Z(t))$ is not bounded. Then we claim that $Y(t)$ is unbounded. If $Y(t)$ is bounded, the first equation of \eq{ds5} would imply
$$
X_{t}+\alpha X(t)\leq \alpha Y^{p}(t)\leq C \quad\mbox{ for all } t\geq 0.
$$
which is equivalent to 
$$
\Big(X(t)e^{\alpha t}\Big)'\leq Ce^{\alpha t} \quad\mbox{ for all } t\geq 0.
$$
Integrating the above inequality we easily deduce 
that $X(t)$ is bounded. Similarly, $Z(t)$ is bounded which contradicts our assumption. Therefore $Y(t)$ must be unbounded. 

\medskip

Let $\tilde{m}>m$ and $(\tilde{v},\tilde{w},\tilde{\psi})$ be the solution of \eq{ds3} with the initial conditions 
$$
\tilde{w}(0)=0,\;\; \tilde{v}(0)=\tilde{m},\;\;\tilde{\psi}(0)=0
$$
defined on the maximum interval $(0,\tilde{R}).$ By Lemma \ref{lemma1.3} we have $\tilde{R}<1$. 

Let $(\tilde{X},\tilde{Y},\tilde{Z})$ be the solution of
\begin{equation}\label{nds5}
 \left\{
\begin{aligned}
& \tilde{X}_{t}+\frac{N-1}{e^{t}-1}\tilde{X}(t)=\alpha(\tilde{Y}^{p}(t)-\tilde{X}(t)) &&\quad\mbox{ for } 0<t<\tilde{T}:=\ln\Big(\frac{1}{1-\tilde{R}}\Big), \\
& \tilde{Y}_{t}=\beta(\tilde{Z}(t)-\tilde{Y}(t)) &&\quad\mbox{ for } 0<t<\tilde T, \\
& \tilde{Z}_{t}+\frac{N-1}{e^{t}-1}\tilde{Z}(t)=\gamma(\tilde{X}^{q}(t)-\tilde{Z}(t)) &&\quad\mbox{ for } 0<t<\tilde T,\\
& \tilde{X}(0)=0,\;\; \tilde{Y}(0)=\tilde{m}/B,\;\; \tilde{Z}(0)=0,
\end{aligned}
\right.
\end{equation}
associated to $(\tilde{v},\tilde{w},\tilde{\psi}).$ Then $(\tilde{X},\tilde{Y},\tilde{Z})$ blows up at $\tilde{T}.$

Since $Y$ is unbounded, we can choose $t_{0}>0$ such that $Y(t_{0})>\tilde{Y}(0)=\tilde{m}/B.$ Let us set
$$
\hat{X}(t)=X(t+t_{0}),\;\; \hat{Y}(t)=Y(t+t_{0}),\;\; \hat{Z}(t)=Z(t+t_{0}).
$$
Then, one can easily check that $(\hat{X}, \hat{Y}, \hat{Z})$ satisfies
$$
\left\{
\begin{aligned}
&\hat X_{t}+\frac{N-1}{e^{t}-1}\hat X(t)\geq\alpha(\hat Y^{p}(t)-\hat X(t)) &&\quad\mbox{ for } 0<t<\infty, \\
&\hat Y_{t}=\beta(\hat Z(t)-\hat Y(t)) &&\quad\mbox{ for } 0<t<\infty, \\
&\hat Z_{t}+\frac{N-1}{e^{t}-1}\hat Z(t)\geq \gamma(\hat X^{q}(t)-\hat Z(t)) &&\quad\mbox{ for } 0<t<\infty,\\
&\hat X(0)>0, \hat Y(0)>\tilde m/B, \hat Z(0)=0.
\end{aligned}
\right.
$$
In virtue of Lemma \ref{lemma1.2} we deduce that
$$
\hat{X}(t)>\tilde{X}(t),\;\; \hat{Y}(t)>\tilde{Y}(t),\;\; \hat{Z}(t)> \tilde{Z}(t)
$$
which contradicts the fact that $(\tilde{X},\tilde{Y},\tilde{Z})$ blows up in finite time.  Hence $\xi(t)=(X(t),Y(t),Z(t))$ is bounded as $t\rightarrow \infty$.

\medskip

\noindent {\it Step 2:} {\it Analysis of the autonomous system associated with \eq{ds5}.} 

We shall embed the autonomous system associated to \eq{ds5} in the whole $\R^{3}$ by considering the initial value problem
\begin{equation}\label{eqna1}
\left\{
\begin{aligned}
&\zeta_{t}= g(\zeta) \quad\mbox{ for all } t\in \R,\\
&\zeta(0)=\zeta_{0}\in \R^{3},
\end{aligned}
\right.
\end{equation}
where 
$$
\zeta= \left(\begin{array}{c} X\\Y\\Z \end{array}\right) ,\;\; g(\zeta)= g\left(\begin{array}{c} X\\Y\\Z \end{array}\right)= \left(\begin{array}{c} \alpha(Y|Y|^{p-1}-X)\\\beta(Z-Y)\\\gamma(X|X|^{q-1}-Z) \end{array}\right).
$$

Using a standard comparison result we have
\begin{lemma}\label{lemmadet11}
Let $(X(t),Y(t),Z(t))$ be the solution of \eq{eqna1} and $t_{0}\in \R.$ Then
\begin{enumerate}
\item [ (i) ] If $X(t_{0}),Y(t_{0}),Z(t_{0})\leq 0$ then  $X(t),Y(t),Z(t)\leq 0$, for all $t\leq t_{0}$.
\item [ (ii) ] If $X(t_{0}),Y(t_{0}),Z(t_{0})\geq 0$ then  $X(t),Y(t),Z(t)\geq 0$, for all $t\geq t_{0}$.
\end{enumerate}
\end{lemma}

The system \eq{eqna1} is cooperative and has negative divergence. It has exactly three equilibria, namely ${\bf0}=(0,0,0)$,\;\;${\bf1}=(1,1,1)$ and ${\bf-1}=(-1,-1,-1).$ It is easy to check that {\bf0} is asymptotically stable. The linearized matrix at {\bf1} and {\bf-1} is
$$
M=\left[\begin{array}{ccc} -\alpha&\alpha p&0 \\ 0&-\beta&\beta \\ \gamma q&0&-\gamma \end{array}\right],
$$
and the eigenvalues $\lambda_{i}$, $1\leq i\leq 3,$ are solutions of
$$
(\lambda+\alpha)(\lambda+\beta)(\lambda+\gamma)-pq\alpha \beta \gamma= 0.
$$
From the definition of $\alpha$, $\beta$ and $\gamma$ in \eq{dsb2}, we have
$$
\alpha+1= p\beta,\;\; \beta+1= \gamma,\;\; \gamma+1= q\alpha.
$$
This shows that $\lambda_{1}=1$ is an eigenvalue of $M.$ Also $\lambda_{2}+\lambda_{3}< 0$ and $\lambda_{2}\lambda_{3}=(pq-1)(\alpha \beta \gamma)>0.$
Thus, $Re(\lambda_{2})< 0$,\;\;$Re(\lambda_{3})< 0.$ So {\bf1},\;\;{\bf-1} are saddle points with two-dimensional stable manifolds. Using Lemma \ref{lemmadet11} and the fact that {\bf0} is asymptotically stable, we deduce that the system \eq{eqna1} has no circuits. By Theorems \ref{thmdet2} and \ref{thmdet3} any compact limit set of \eq{eqna1} reduces to an equilibrium point. Hence, any bounded trajectory $\phi(t,\zeta)$ converges both backward and forward in time to one of the three equilibria described above.

\medskip

\noindent {\it Step 3: Analysis of the non-autonomous system \eq{ds5}.}

\medskip

Let $\xi_{0}=(0,m/B,0)$ and denote by $\Phi(\cdot,\xi_{0})$ the semiflow associated to \eq{ds5}. By Theorem \ref{thmfd2}, the $\omega$-limit set $\omega_{\Phi}(\xi_{0})$ is invariant under the flow $\phi$ of the autonomous system \eq{eqna1}. Thus
$$
\phi(t,\omega_{\Phi}(\xi_{0}))=\omega_{\Phi}(\xi_{0}) \quad\mbox{ for all } t\geq 0.
$$
Due to the group property of the flow $\phi$, the above equality is true for all $t\in \R$ because
\begin{equation}\label{nau1}
\omega_{\Phi}(\xi_{0})= \phi(0,\omega_{\Phi}(\xi_{0}))=\phi(-|t|,\phi(|t|,\omega_{\Phi}(\xi_{0}))=\phi(-|t|,\omega_{\Phi}(\xi_{0})).
\end{equation}

Hence, $\phi(t,\omega_{\Phi}(\xi_{0}))= \omega_{\Phi}(\xi_{0})$ for all $t\in \R.$

\medskip

Let $z\in \omega_{\Phi}(\xi_{0})$. Since  $\omega_{\Phi}(\xi_{0})$ is chain recurrent, for all $n\geq 1$ there exist a finite sequence of points in  $\omega_{\Phi}(\xi_{0})$
$$
z=\zeta_1,\;\zeta_2,\;\dots,\; \zeta_{k_n},\;\zeta_{k_n+1}=z
$$
and a sequence of finite times
$$
t_1,\;t_2,\;\dots,\; t_{k_n}\geq n
$$
such that
$$
|\phi(t_i,\zeta_i)-\zeta_{i+1}|<\frac{1}{n}\quad\mbox{ for all }1\leq i\leq k_n.
$$
In particular, for $i=k_n$ we find
\begin{equation}\label{eqzeta}
|\phi(t_{k_n},\zeta_{k_n})-z|<\frac{1}{n}\quad\mbox{ for all }n\geq 1.
\end{equation}
Since $\{\zeta_{k_n}\}\subset \omega_{\Phi}(\xi_{0})$ is bounded, it follows that up to a subsequence (still denoted by $\{\zeta_{k_n}\}$) we have $\zeta_{k_n}\to \zeta_0$ as $n\to \infty$, for some $\zeta_0\in \omega_{\Phi}(\xi_{0})$. By Step 2, $\phi(t,\zeta_0)\to \ell\in \{{\bf 0,1}\}$ as $t\to \infty$.  Using the continuous dependence of the flow $\phi$ on the initial data we can let $n\to \infty$ in \eq{eqzeta} to deduce $z=\ell \in \{{\bf 0,1}\}$. Thus, $\omega_{\Phi}(\xi_{0})\subset \{\bf0,\bf1\}$.
Since $\omega_{\Phi}(\xi_{0})$ is connected, by Theorem \ref{thmfd2}(a) it follows that $\omega_{\Phi}(\xi_{0})=\{\bf0\}$ or $\omega_{\Phi}(\xi_{0})=\{\bf1\}.$ Assume by contradiction that $\omega_{\Phi}(\xi_{0})=\{\bf0\}.$ Then $(X(t),Y(t),Z(t))$ tends to {\bf0} as $t\rightarrow \infty.$ Then, we may find $t_{0}>0$ such that
$$
0<X(t_{0}), Y(t_{0}), Z(t_{0})< 1.
$$
Let now $\tilde{m}>m$ and $(\tilde{X}(t),\tilde{Y}(t),\tilde{Z}(t))$ be the solution of \eq{nds5}. By taking $\tilde{m}$ close to $m$ and using the continuous dependence on the initial data of solution to \eq{ds5}, we may assume
$$
0< \tilde{X}(t_{0}), \tilde{Y}(t_{0}), \tilde{Z}(t_{0})< 1.
$$
A comparison principle now implies 
$$
0< \tilde{X}(t), \tilde{Y}(t), \tilde{Z}(t)< 1 \quad\mbox{ for all } t\geq t_{0}.
$$
But the above inequalities contradict the fact that $(\tilde{X}(t),\tilde{Y}(t),\tilde{Z}(t))$ blows up in finite time. Hence, $\xi(t)=(X(t),Y(t),Z(t))\rightarrow {\bf1}.$

\medskip

Using \eq{eqhet1} it follows,
\begin{equation}\label{end1}
\left\{
\begin{aligned}
&v(r)(1-r)^{\beta}\rightarrow B \quad\mbox{ as } r\nearrow 1,\\
&w(r)(1-r)^{\alpha}\rightarrow A \quad\mbox{ as } r\nearrow 1,\\
&\psi(r)(1-r)^{\gamma}\rightarrow C \quad\mbox{ as } r\nearrow 1.
\end{aligned}
\right.
\end{equation}
Now, the first part of equation \eq{end1} implies \eq{blowu0}. Let $\varepsilon>0$. Then there exist $\delta \in (0,1)$ such that
\begin{equation}\label{end2}
(1-\varepsilon)A(1-r)^{-\alpha}\leq w(r)=u'(r)\leq (1+\varepsilon)A(1-r)^{-\alpha} \quad\mbox{ for all } r\in [\delta,1).
\end{equation}
Assume $q>2(1+1/p)$. Thus $\alpha \in (0,1)$. By Corollary \ref{corblowup3}, $u$ is bounded and increasing on $(0,1).$ Thus there exists $L=\lim_{r\nearrow 1}u(r)$ and integrating \eq{end2} over $[r,1]$, we find
$$
\frac{A}{1-\alpha}(1-\varepsilon)\leq \frac{L-u(r)}{(1-r)^{1-\alpha}}\leq \frac{A}{1-\alpha}(1+\varepsilon) \quad\mbox{ for all } r\in [\delta,1).
$$
This proves part (i) of Theorem \ref{thm3}.

\medskip

Assume now $q=2(1+1/p)$. Thus $\alpha=1.$ Integrating \eq{end2} over $[\delta,r]$, where $\delta<r<1$, we find
$$
A(1-\varepsilon)\leq \liminf_{r\nearrow 1}\frac{u(r)}{\ln(\frac{1}{1-r})}\leq \limsup_{r\nearrow 1}\frac{u(r)}{\ln(\frac{1}{1-r})}\leq A(1+\varepsilon).
$$
Letting $\varepsilon \rightarrow 0,$ we get 
$$
\lim_{r\nearrow 1}\frac{u(r)}{\ln(\frac{1}{1-r})}=A.
$$
This proves part (ii) of Theorem \ref{thm3}.

\medskip

Assuming $q<2(1+1/p)$, in a similar way as before we derive the proof of part (iii) in Theorem \ref{thm3}.

\section{Proof of the Theorem \ref{thmrn}}
As in the proof of  Step 1 in Theorem \ref{thm1}, we obtain that $u'$, $v'$, $u$, $v$ are increasing and
\begin{equation*}
\left\{
\begin{aligned}
&u'(r)=r^{1-N}\int_{0}^{r}{s^{N-1}v^{p}(s)}ds \quad\mbox{ for all } r>0,\\
&v'(r)=r^{1-N}\int_{0}^{r}{s^{N-1}(u')^{q}(s)}ds \quad\mbox{ for all } r>0.
\end{aligned}
\right.
\end{equation*}
This yields
\begin{equation}\label{rn1}
\frac{rv^{p}(0)}{N}\leq u'(r)\leq \frac{rv^{p}(r)}{N} \quad\mbox{ for all } r>0
\end{equation}
and 
\begin{equation}\label{rn2}
\frac{v^{pq}(0)r^{1+q}}{N^{q}(N+q)}\leq v'(r)\leq \frac{ru'^{q}(r)}{N} \quad\mbox{ for all } r>0.
\end{equation}
From \eq{rn1} and \eq{rn2} we deduce that $u'(r)$, $v'(r)$, $u(r)$, $v(r)$ tend to infinty as $r\rightarrow \infty$.

Inspired by the change of variables introduced in \cite{HV1996} (see also \cite{BVH2010,G2012}) we define 
$$
X(t)= \frac{ru'(r)}{u(r)},\;\;Y(t)= \frac{rv'(r)}{v(r)},\;\;Z(t)=\frac{rv^{p}(r)}{u'(r)} \mbox{ and }W(t)=\frac{ru'^{q}(r)}{v'(r)},
$$
where $t= \ln(r)$ for  $r\in (0,\infty)$. A direct calculation shows that $(X(t),Y(t),Z(t),W(t))$ satisfies
\begin{equation}\label{rn3}
\left\{
\begin{aligned}
&X_{t}= X(Z-(N-2)-X) \quad\mbox{ for all } t\in \R, \\
&Y_{t}= Y(W-(N-2)-Y) \quad\mbox{ for all } t\in \R, \\
&Z_{t}= Z(N+pY-Z) \quad\mbox{ for all } t\in \R, \\
&W_{t}= W(qZ-qN+q+N-W) \quad\mbox{ for all } t\in \R.
\end{aligned}
\right.
\end{equation}
By L'Hopital's rule we have $\lim_{t\rightarrow \infty}X(t)=2-N+\lim_{t\rightarrow \infty}Z(t)$. Thus, it is enough to consider the last three equations of \eq{rn3}, namely
\begin{equation}\label{rn4}
\left\{
\begin{aligned}
&Y_{t}= Y(W-(N-2)-Y) \quad\mbox{ for all } t\in \R, \\
&Z_{t}= Z(N+pY-Z) \quad\mbox{ for all } t\in \R, \\
&W_{t}= W(qZ-qN+q+N-W) \quad\mbox{ for all } t\in \R.
\end{aligned}
\right.
\end{equation}
We rewrite our system as
\begin{equation}\label{rn5}
\zeta_{t}= g(\zeta) 
\end{equation}
where 
$$
\zeta=\left(\begin{array}{c}Y(t)\\Z(t)\\W(t)\end{array}\right) \quad\mbox{ and } \quad g(\zeta)= \left(\begin{array}{c} Y(W-(N-2)-Y)\\Z(N+pY-Z)\\ W(qZ-qN+q+N-W) \end{array}\right).
$$

Since the system \eq{rn5} is cooperative, the following comparison principle holds:
\begin{lemma}\label{lrn1}
Let $\zeta(t)= \left(\begin{array}{c}Y(t)\\Z(t)\\W(t)\end{array}\right)$ and $\tilde{\zeta}(t)= \left(\begin{array}{c}\tilde{Y}(t)\\ \tilde{Z}(t)\\ \tilde{W}(t)\end{array}\right)$ be two nonnegative solutions of \eq{rn5} such that
$$
Y(t_0)\geq \tilde{Y}(t_0), \;\;\;Z(t_0)\geq \tilde{Z}(t_0),\;\;\; W(t_0)\geq \tilde{W}(t_0)
$$
for some $t_0\in \R$. 
Then
$$
Y(t)\geq \tilde{Y}(t), \;\;\;Z(t)\geq \tilde{Z}(t),\;\;\; W(t)\geq \tilde{W}(t) \quad\mbox{ for all }t\geq t_{0}.
$$
\end{lemma}

\bigskip

\bigskip

From \eq{rn1} and \eq{rn2} we have $Z\geq N$ and $W\geq N$. Therefore there are only two equilibria of \eq{rn4} which satisfy $Z\geq N$ and $W\geq N$, namely 
$$
\zeta_{1}= \left(\begin{array}{c}0\\N\\N+q\end{array}\right) \quad\mbox{  and }
\zeta_{2}=\left(\begin{array}{c}\frac{2+q}{1-pq}\\N+\frac{p(2+q)}{1-pq}\\N-2+\frac{2+q}{1-pq}\end{array}\right).
$$
\begin{lemma}\label{lrn2}
$\zeta_{2}$ is asymptotically stable.
\end{lemma}
\begin{proof}
The linearized matrix at $\zeta_2$ is
$$
M=\left[\begin{array}{ccc}-Y_2&0&Y_2\\pZ_2&-Z_2&0\\0&qW_2&-W_2\end{array}\right].
$$
The characteristic polynomial of $M$ is
$$
P(\lambda)=\det(\lambda I-M)=\lambda_{3}+\alpha\lambda_{2}+\beta\lambda+(1-pq)\gamma
$$
where 
\begin{equation*}
\begin{aligned}
\alpha&=Y_2+Z_2+W_2\\
\beta&=Y_2Z_2+Y_2W_2+Z_2W_2\\
\gamma&=Y_2Z_2W_2.
\end{aligned}
\end{equation*}
Since $\alpha$, $\beta$, $\gamma>0$ and $pq<1$, we have $P(\lambda)>0$ for all $\lambda \geq 0$. If $P$ has three real roots then they are all negative, so $\zeta_2$ is asymptotically stable in this case. It remains to consider the situation where $P$ has exactly one real root. Let $\lambda_1 \in \R$ and $\lambda_2,\lambda_3 \in \mathbb{C}\setminus \R$ be the roots of $P$. We claim that ${\rm Re}(\lambda_2)={\rm Re}(\lambda_3)<0$. We need to show $P(-\alpha)=-\beta\alpha+(1-pq)\gamma<0$, that is, $\beta\alpha>(1-pq)\gamma$. By AM-GM inequality we find
$$
\alpha \geq 3\sqrt[3]{Y_2Z_2W_2}\quad  \mbox{ and } \quad \beta \geq 3\sqrt[3]{(Y_2Z_2W_2)^{2}}
$$
which yields $\alpha\beta> (1-pq)\gamma$. Hence $\zeta_2$ is asymptotically stable.

\end{proof}

\begin{lemma}\label{lrn3}
For all $t\in \R$, we have
\begin{equation}\label{rn6}
0\leq Y(t)\leq \frac{2+q}{1-pq},
\end{equation}
\begin{equation}\label{rn7}
N\leq Z(t)\leq N+\frac{p(2+q)}{1-pq},
\end{equation}
\begin{equation}\label{rn8}
N+q\leq W(t)\leq N-2+\frac{2+q}{1-pq}.
\end{equation}
\end{lemma}
\begin{proof}
Since $v'(0)=0$ and $v(0)=0$ we obtain $\lim_{t\rightarrow -\infty}Y(t)= \lim_{r\rightarrow 0}\frac{rv'(r)}{v(r)}=0$. Also
\begin{equation}\label{crn1}
u''(r)+\frac{N-1}{r}u'(r)=v^{p}(r) \quad\mbox{ for all } r>0.
\end{equation}
Using L'Hopital's rule, we deduce that
$$
\lim_{r\rightarrow 0}\frac{u'(r)}{r}=\lim_{r\rightarrow 0}u''(r)= u''(0)
$$
which combined with \eq{crn1} yields
$$
u''(0)= \frac{v^{p}(0)}{N}=\lim_{r\rightarrow 0}\frac{u'(r)}{r}.
$$
Therefore
\begin{equation*}
\lim_{t\rightarrow -\infty}Z(t)=\lim_{r\rightarrow 0}\frac{rv^{p}(r)}{u'(r)}=N.
\end{equation*}
We claim that there exists $t_{j}\rightarrow -\infty $ such that
\begin{equation}\label{rn9}
\left\{
\begin{aligned}
&Y(t_{j})\leq Y_2,\\
&Z(t_j)\leq Z_2,\\
&W(t_j)\leq W_2.
\end{aligned}
\right.
\end{equation}
Because $\lim_{t\rightarrow -\infty}Y(t)=0$ and $\lim_{t\rightarrow -\infty}Z(t)= N$, it remains only to prove the last part of \eq{rn9}. Assume by contradiction that this is not true. Thus $W> W_2$ in $(-\infty,t_0)$ for some $t_0\in \R$. Then, by taking $t_0$ small enough we have
$$
W_t= W(qZ-qN+q+N-W)< 0 \quad\mbox{ in }(-\infty,t_0).
$$
Hence, $W$ is decreasing in the neighbourhood of $-\infty$ and so, there exists $\ell= \lim_{t\rightarrow -\infty}W(t)$. Again using L'Hopital's rule we have
\begin{equation*}
\begin{aligned}
\ell=\lim_{t\rightarrow -\infty}W(t)&=\lim_{r\rightarrow 0}\frac{ru'^{q}(r)}{v'(r)}\\
&=\lim_{r\rightarrow 0}\frac{rqu'^{q-1}(r)u''(r)+u'^{q}(r)}{v''(r)}\\
&=\lim_{r\rightarrow 0}\frac{rqu'^{q-1}(r)v^{p}(r)+(1+q-qN)u'^{q}(r)}{u'^{q}(r)-\frac{N-1}{r}v'(r)}\\
&=\lim_{t\rightarrow -\infty}\frac{qZ(t)+1+q-qN}{1-\frac{N-1}{W(t)}}\\
&=\frac{1+q}{1-\frac{N-1}{\ell}}.
\end{aligned}
\end{equation*}
This yields $\ell=N+q<W_2$ which contradicts our assumption that $W>W_{2}$ in a neighbourhood of $-\infty$. This proves the last part of \eq{rn9}. We then apply the Comparison Lemma \ref{lrn1} on all the intervals $[t_j,\infty)$ for $j\geq 1$ to obtain the upper bound inequalities in Lemma \ref{lrn3}. In the same way we obtain the lower bound inequalities.
\end{proof}

\medskip

Let $K=\overline{[[\zeta_{1},\zeta_{2}]]}\subset \R^{3}$. By Lemma \ref{lrn3} we have $\omega(\zeta)\subseteq K $. Since $\zeta_2$ is asymptotically stable, $K$ has no circuits. Also, by \eq{div} we have
 $$
{\rm div}\, g(\zeta)=-W+(q-2)Z+(p-2)Y+N+2-qN+q<0\quad\mbox{ in } K.
$$ 
Using Theorems \ref{thmdet2} and \ref{thmdet3} we deduce that $\omega(\zeta)$ reduces to one of the equilibria $\zeta_1$ or $\zeta_2$. If $\zeta(t)\rightarrow \zeta_1$ as $t\rightarrow \infty$ this implies in particular that $Y(t)\rightarrow 0$ as $t\rightarrow \infty$. On the other hand, from the second equation of \eq{rn3} we deduce $Y_t> 0$ in a neighbourhood of infinity which is impossible given that $Y(t)> 0$ in $\R$. Hence $\zeta(t)\rightarrow \zeta_2$ as $t\rightarrow \infty$, that is 
\begin{equation*}
\begin{aligned}
\lim_{t\rightarrow \infty}X(t)&= 2+\frac{p(2+q)}{1-pq},\\
\lim_{t\rightarrow \infty}Y(t)&= \frac{2+q}{1-pq},\\
\lim_{t\rightarrow \infty}Z(t)&= N+\frac{p(2+q)}{1-pq},\\
\lim_{t\rightarrow \infty}W(t)&= N-2+\frac{2+q}{1-pq}.
\end{aligned}
\end{equation*}
Using $(X(t),Y(t),Z(t),W(t))$ , we have 
\begin{equation*}
\begin{aligned}
\lim_{r\rightarrow \infty}\frac{v(r)}{r^{\frac{q+2}{1-pq}}}&=\lim_{t\rightarrow \infty}(Y(t)Z^{q}(t)W(t))^{\frac{1}{pq-1}}\\
&=(Y_{2}Z^{q}_{2}W_{2})^{\frac{1}{pq-1}}\\
&=\Big[\frac{(2+q)(N(1-pq)+p(2+q))^{q}(N(1-pq)+q(2p+1))}{(1-pq)^{2+q}}\Big]^{\frac{1}{pq-1}}.
\end{aligned}
\end{equation*}
And
\begin{equation*}
\begin{aligned}
\lim_{r\rightarrow \infty}\frac{u(r)}{r^{\frac{2+2p-pq}{1-pq}}}&=\lim_{t\rightarrow \infty}\frac{(Y(t)Z^{q}(t)W(t))^{\frac{p}{pq-1}}}{X(t)Z(t)}\\
&=\frac{(Y_{2}Z^{q}_{2}W_{2})^{\frac{p}{pq-1}}}{X_{2}Z_{2}}\\
&=\Big[\frac{(2+q)(N(1-pq)+p(2+q))^{1/p}(N(1-pq)+q(2p+1))}{(1-pq)^{\frac{2p+2-pq}{p}}(2+2p-pq)^{\frac{pq-1}{p}}}\Big]^{\frac{p}{pq-1}}.
\end{aligned}
\end{equation*}

\noindent{\bf Acknowledgement.} This work is part of the author's PhD thesis and has been carried out with the financial support of the Research Demonstratorship Scheme offered by the School of Mathematical Sciences, University College Dublin.

\end{document}